\documentclass{article}
\pdfoutput=1
\usepackage{amsmath}
\usepackage{amssymb}
\usepackage{amsfonts}
\usepackage{pgf}

\pgfdeclareimage[height=8cm]{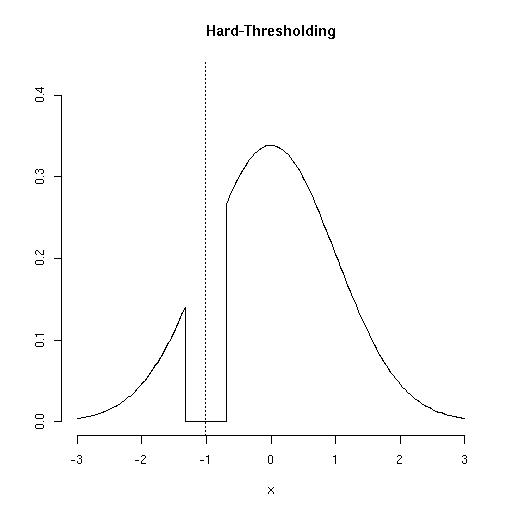}{figure1}
\pgfdeclareimage[height=8cm]{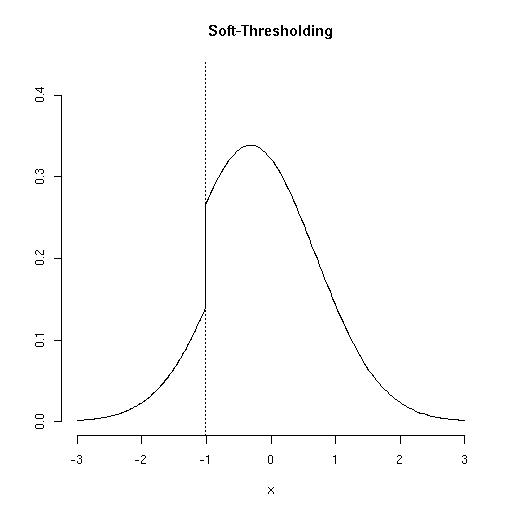}{figure2}
\pgfdeclareimage[height=8cm]{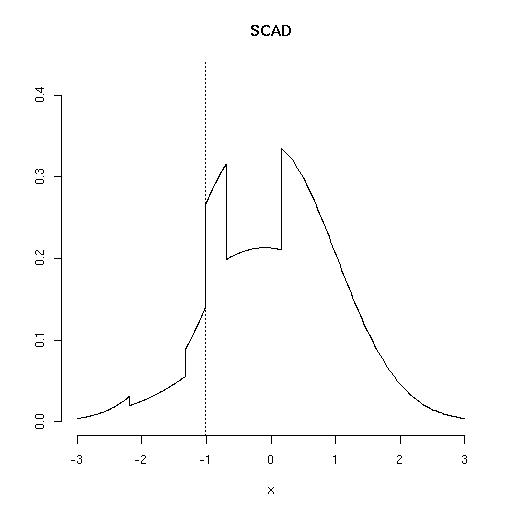}{figure3}
\newtheorem{theorem}{Theorem}

\newtheorem{proposition}[theorem]{Proposition}
\newtheorem{remark}[theorem]{Remark}

\newenvironment{proof}[1][Proof]{\noindent\textbf{#1.} }{\ \rule{0.5em}{0.5em}}
\input{tcilatex}
\begin{document}

\title{On the Distribution of Penalized Maximum Likelihood Estimators: The
LASSO, SCAD, and Thresholding}
\author{Benedikt M. P\"{o}tscher \\
Department of Statistics, University of Vienna\\
and \and Hannes Leeb \\
Department of Statistics, Yale University}
\date{First version: October 2007\\
This version: March 2009}
\maketitle

\begin{abstract}
We study the distributions of the LASSO, SCAD, and thresholding estimators,
in finite samples and in the large-sample limit. The asymptotic
distributions are derived for both the case where the estimators are tuned
to perform consistent model selection and for the case where the estimators
are tuned to perform conservative model selection. Our findings complement
those of Knight and Fu (2000) and Fan and Li (2001). We show that the
distributions are typically highly nonnormal regardless of how the estimator
is tuned, and that this property persists in large samples. The uniform
convergence rate of these estimators is also obtained, and is shown to be \
slower than $n^{-1/2}$ in case the estimator is tuned to perform consistent
model selection. An impossibility result regarding estimation of the
estimators' distribution function is also provided.


\end{abstract}

\section{Introduction}

Penalized maximum likelihood estimators have been studied intensively in the
last few years. A prominent example is the least absolute selection and
shrinkage (LASSO) estimator of Tibshirani (1996). Related variants of the
LASSO include the Bridge estimators studied by Frank and Friedman (1993),
least angle regression (LARS) of Efron, Hastie, Johnston, Tibshirani (2004),
or the smoothly clipped absolute deviation (SCAD) estimator of Fan and Li
(2001). Other estimators that fit into this framework are hard- and
soft-thresholding estimators. While many properties of penalized maximum
likelihood estimators are now well understood, the understanding of their
distributional properties, such as finite-sample and large-sample limit
distributions, is still incomplete. The probably most important contribution
in this respect is Knight and Fu (2000) who study the asymptotic
distribution of the LASSO estimator (and of Bridge estimators more
generally) when the tuning parameter governing the influence of the penalty
term is chosen so that the LASSO acts as a conservative model selection
procedure (that is, a procedure that does not select underparameterized
models asymptotically, but selects overparameterized models with positive
probability asymptotically); see also Knight (2008). In Knight and Fu
(2000), the asymptotic distribution is obtained in a fixed-parameter as well
as in a standard local alternatives setup. This is complemented by a result
in Zou (2006) who considers the fixed-parameter asymptotic distribution of
the LASSO when tuned to act as a consistent model selection procedure.
Another contribution is Fan and Li (2001) who derive the asymptotic
distribution of the SCAD estimator when the tuning parameter is chosen so
that the SCAD estimator performs consistent model selection; in particular,
they establish the so-called `oracle' property for this estimator. The
results in that latter paper are also fixed-parameter asymptotic results. It
is well-known that fixed-parameter (i.e., pointwise) asymptotic results can
give a wrong picture of the estimators' actual behavior, especially when the
estimator performs model selection; see, e.g., Kabaila (1995), or Leeb and P%
\"{o}tscher (2005, 2008a). Therefore, it is interesting to take a closer
look at the actual distributional properties of such estimators.

In the present paper we study the finite-sample as well as the asymptotic
distributions of the hard-thresholding, the LASSO (which coincides with
soft-thresholding in our context), and the SCAD estimator. We choose a model
that is simple enough to facilitate an explicit finite-sample analysis that
showcases the strengths and weaknesses of these estimators in a readily
accessible framework. Yet, the model considered here is rich enough to
demonstrate a variety of phenomena that will also occur in more complex
models. We study both the cases where the estimators are tuned to perform
conservative model selection as well as where the tuning is such that the
estimators perform consistent model selection. We find that the
finite-sample distributions can be decisively non-normal (e.g., multimodal).
Moreover, we find that a fixed-parameter asymptotic analysis gives highly
misleading results. In particular, the `oracle' property, which is based on
a fixed-parameter asymptotic analysis, is shown to not provide a reliable
assessment of the estimators' actual performance. For these reasons, we also
obtain the asymptotic distributions of the estimators mentioned before in a
general `moving parameter' asymptotic framework, which better captures
essential features of the finite-sample distribution. [Interestingly, it
turns out that in the consistent model selection case a `moving parameter'
asymptotic framework more general than the usual $n^{-1/2}$-local asymptotic
framework is necessary to exhibit the full range of possible limiting
distributions.] Furthermore, we derive the uniform convergence rate of the
estimators and show that it is slower than $n^{-1/2}$ in the case where the
estimators are tuned to perform consistent model selection. This again
exposes the misleading character of the `oracle' property. We also show that
the finite-sample distribution of these estimators can not be estimated in
any reasonable sense, complementing results of this sort in the literature
(Leeb and P\"{o}tscher (2006a,b, 2008b), P\"{o}tscher (2006)). In a
subsequent paper, P\"{o}tscher and Schneider (2009), analogous results are
obtained for the adaptive LASSO estimator.

We note that penalized maximum likelihood estimators are intimately related
to more classical post-model-selection estimators. The distributional
properties of the latter estimators have been studied by Sen (1979), P\"{o}%
tscher (1991), and Leeb and P\"{o}tscher (2003, 2005, 2006a,b, 2008b).

The paper is organized as follows: The model and the estimators are
introduced in Section \ref{mo_est}, and the model selection probabilities
are discussed in Section \ref{prob}. Consistency, uniform consistency, and
uniform convergence rates of the estimators are the subject of Section \ref%
{cons&ucons}. The finite-sample distributions are derived in Section \ref%
{finite}, whereas the asymptotic distributions are studied in Section \ref%
{asydistribs}. Section \ref{imposs} provides impossibility results
concerning the estimation of the finite-sample distributions of the
estimators, and Section~\ref{conclusion} concludes and summarizes our main
findings. The appendix contains results on the asymptotic distribution in
the consistent model selection case when the estimators are scaled by the
inverse of the uniform convergence rate obtained in Section \ref{cons&ucons}
rather than by $n^{1/2}$.

\section{The Model and the Estimators\label{mo_est}}

We start with the orthogonal linear regression model%
\begin{equation*}
Y=X\beta +u
\end{equation*}%
where $X^{\prime }X$ is diagonal and the vector $u$ is multivariate normal
with mean zero and variance covariance matrix $\sigma ^{2}I$. The
multivariate linear model with orthogonal design occurs in many important
settings, including wavelet regression or the analysis of variance. Because
we consider penalized least-squares estimators with a penalty term that is
separable with respect to $\beta $, the resulting estimators for the
components of $\beta $ are mutually independent and each component estimator
is equivalent to the corresponding penalized least squares estimator in a
univariate Gaussian location model. We therefore restrict attention to this
simple model in the sequel without loss of generality.

Suppose $y_{1},\ldots ,y_{n}$ are independent and each distributed as $%
N(\theta ,\sigma ^{2})$. We assume for simplicity that $\sigma ^{2}$ is
known, and hence we can set $\sigma ^{2}=1$ without loss of generality.
Apart from the standard maximum likelihood (least squares) estimator $\bar{y}
$ we consider the following estimators:

\begin{enumerate}
\item The hard-thresholding estimator $\hat{\theta}_{H}=\bar{y}\boldsymbol{1}%
(\left\vert \bar{y}\right\vert >\eta _{n})$ where the threshold $\eta _{n}$
is a positive real number and $\boldsymbol{1}(\cdot )$ denotes the indicator
function. The threshold $\eta _{n}$ is a tuning parameter set by the user.
The hard-thresholding estimator can be viewed as a penalized least-squares
estimator that arises as the solution to the minimization problem\footnote{%
The penalty corresponding to hard thresholding given in Fan and Li (2001)
differs from the correct one that we use here, because of a scaling error in
equations (2.3) and (2.4) of Fan and Li (2001).} 
\begin{equation*}
\dsum\limits_{t=1}^{n}(y_{t}-\theta )^{2}\;\;+\;\; n\Big( \eta
_{n}^{2}-(\left\vert \theta \right\vert -\eta _{n})^{2}\boldsymbol{1}%
(\left\vert \theta \right\vert <\eta _{n})\Big) .
\end{equation*}%
We also note here that for $\eta _{n}=n^{-1/4}$ the hard-thresholding
estimator is a simple instance of Hodges' estimator (see, e.g., Lehmann and
Casella (1998), pp. 440-443).

\item The soft-thresholding estimator $\hat{\theta}_{S}=\limfunc{sign}(\bar{y%
})(\left\vert \bar{y}\right\vert -\eta _{n})_{+}$ with $\eta _{n}$ as
before. [Here $\limfunc{sign}(x)$ is defined as $-1$, $0$, and $1$ in case $%
x<0$, $x=0$, and $x>0$, respectively, and $z_{+}$ is shorthand for $\max
\{z,0\}$.] That estimator arises as the solution to the penalized
least-squares problem%
\begin{equation*}
\dsum\limits_{t=1}^{n}(y_{t}-\theta )^{2}+2n\eta _{n}\left\vert \theta
\right\vert
\end{equation*}%
which shows that $\hat{\theta}_{S}$ coincides with the LASSO in the form
considered in Knight and Fu (2000). Note that the tuning parameter in the
latter reference is $\lambda _{n}=2n\eta _{n}$.

\item The SCAD-estimator of Fan and Li (2001) is -- in the present context
-- given by 
\begin{equation*}
\hat{\theta}_{SCAD}=\left\{ 
\begin{array}{ll}
\limfunc{sign}(\bar{y})(\left\vert \bar{y}\right\vert -\eta _{n})_{+} & 
\text{if \ \ }\left\vert \bar{y}\right\vert \leq 2\eta _{n}, \\ 
\left\{ (a-1)\bar{y}-\limfunc{sign}(\bar{y})a\eta _{n}\right\} /(a-2) & 
\text{if \ \ }2\eta _{n}<\left\vert \bar{y}\right\vert \leq a\eta _{n}, \\ 
\bar{y} & \text{if \ \ }\left\vert \bar{y}\right\vert >a\eta _{n},%
\end{array}%
\right.
\end{equation*}%
where $a>2$ is an additional tuning parameter. This estimator can be viewed
as a simple combination of soft-thresholding for `small' $\left\vert \bar{y}%
\right\vert $ and hard-thresholding for `large' $\left\vert \bar{y}%
\right\vert $, with a (piecewise) linear interpolation in-between.
Alternatively, the estimator can also be obtained as a solution to a
penalized least squares problem; see Fan and Li (2001) for details. We note
that the SCAD-estimator is closely related to the firm shrinkage estimator
of Bruce and Gao (1996).
\end{enumerate}

\section{Model Selection Probabilities\label{prob}}

Each of the three estimators discussed above induces a selection between the
restricted model $M_{R}$ consisting only of the $N(0,1)$-distribution and
the unrestricted model $M_{U}=\{N(\theta ,1):\theta \in \mathbb{R}\}$ in an
obvious way, i.e., $M_{R}$ is selected if the respective estimator for $%
\theta $ equals zero, and $M_{U}$ is selected otherwise. In the present
context, the hard-thresholding estimator $\hat{\theta}_{H}$ is furthermore
nothing else than a traditional pre-test estimator that chooses between the
unrestricted maximum likelihood estimator $\hat{\theta}_{U}=\bar{y}$ and the
restricted maximum likelihood estimator $\hat{\theta}_{R}\equiv 0$ according
to the outcome of a $t$-type test for the hypothesis $\theta =0$.

We now study the model selection probabilities, i.e., the probabilities that
model $M_{U}$ or $M_{R}$, respectively, is selected. As they add up to one,
it suffices to consider one of them. First note that the probability\ of
selecting the model $M_{R}$ is the same for each of the estimators $\hat{%
\theta}_{H}$, $\hat{\theta}_{S}$, and $\hat{\theta}_{SCAD}$ (provided the
same tuning parameter $\eta _{n}$ is used). This is so because the events $\{%
\hat{\theta}_{H}=0\}$, $\{\hat{\theta}_{S}=0\}$, and $\{\hat{\theta}%
_{SCAD}=0\}$ coincide. Hence, 
\begin{equation}
\begin{split}
P_{n,\theta }(\hat{\theta}=0)& =P_{n,\theta }(\left\vert \bar{y}\right\vert
\leq \eta _{n})=\Pr \left( \left\vert Z+n^{1/2}\theta \right\vert \leq
n^{1/2}\eta _{n}\right) \\
& =\Phi (n^{1/2}(-\theta +\eta _{n}))-\Phi (n^{1/2}(-\theta -\eta _{n})),
\end{split}
\label{model_prob}
\end{equation}%
where $\hat{\theta}$ stands for any of the estimators $\hat{\theta}_{H}$, $%
\hat{\theta}_{S}$, and $\hat{\theta}_{SCAD}$, and where $Z$ is a standard
normal random variable with cumulative distribution function (cdf) $\Phi $.
Here we use $P_{n,\theta }$ to denote the probability governing a sample of
size $n$ when $\theta $ is the true parameter, and $\Pr $ to denote a
generic probability measure.

In the following we shall always impose the condition that $\eta
_{n}\rightarrow 0$ for asymptotic considerations, which guarantees that the
probability of incorrectly selecting the restricted model $M_{R}$ (i.e.,
selecting $M_{R}$ if the true $\theta $ is non-zero) vanishes
asymptotically. Conversely, if this probability vanishes asymptotically for 
\textit{every} $\theta \neq 0$, then $\eta _{n}\rightarrow 0$ follows.
Therefore, the condition $\eta _{n}\rightarrow 0$ is a basic one and without
this condition the estimators $\hat{\theta}_{H}$, $\hat{\theta}_{S}$, and $%
\hat{\theta}_{SCAD}$ do not seem to be of much interest (from an asymptotic
viewpoint). As we shall see in the next section, this basic condition is
also equivalent to consistency for $\theta $ of the hard-thresholding
(soft-thresholding, SCAD) estimator.

Given the condition $\eta _{n}\rightarrow 0$, two cases need to be
distinguished: (i) $n^{1/2}\eta _{n}\rightarrow e$, $0\leq e<\infty $ and
(ii) $n^{1/2}\eta _{n}\rightarrow e=\infty $.\footnote{%
There is no loss in generality here in the sense that the general case where
only $\eta _{n}\rightarrow 0$ holds can always be reduced to case (i) or
case (ii) by passing to subsequences.} In case (i) the hard-thresholding
(soft-thresholding, SCAD) estimator acts as a conservative model selection
procedure, i.e., the probability of selecting the unrestricted model $M_{U}$
has a positive limit even when $\theta =0$, whereas in case (ii) it acts as
a consistent model selection procedure, i.e., this probability vanishes in
the limit when $\theta =0$. This is immediately seen by inspection of (\ref%
{model_prob}). These facts have long been known, see Bauer, P\"{o}tscher,
and Hackl (1988).

The results discussed in the preceding paragraph are of a `pointwise'
asymptotic nature in the sense that the value of $\theta $ is held fixed
when sample size $n$ goes to infinity. As noted before, such pointwise
asymptotic results often miss essential aspects of the finite-sample
behavior, especially in the context of model selection; cf. Leeb and P\"{o}%
tscher (2005). To obtain large-sample results that better capture
finite-sample phenomena, we next present a `moving parameter' asymptotic
analysis, i.e., we allow $\theta $ to vary with $n$ as $n\rightarrow \infty $%
. The following result shows in particular that convergence of the model
selection probability to its limit in a pointwise asymptotic analysis is 
\emph{not} uniform in $\theta \in \mathbb{R}$ (in fact, it fails to be
uniform in any neighborhood of $\theta =0$).

\begin{proposition}
\label{selection_prob}Let $\hat{\theta}$ be either $\hat{\theta}_{H}$, $\hat{%
\theta}_{S}$, or $\hat{\theta}_{SCAD}$. Suppose that $\eta _{n}\rightarrow 0$
and $n^{1/2}\eta _{n}\rightarrow e$ with $0\leq e\leq \infty $. \newline
(i) Assume $e<\infty $ (corresponding to conservative model selection).
Suppose the true parameter $\theta _{n}\in \mathbb{R}$ satisfies $%
n^{1/2}\theta _{n}\rightarrow \nu \in \mathbb{R}\cup \{-\infty ,\infty \}$.
Then 
\begin{equation*}
\lim_{n\rightarrow \infty }P_{n,\theta _{n}}(\hat{\theta}=0)=\Phi (-\nu
+e)-\Phi (-\nu -e).
\end{equation*}%
(ii) Assume $e=\infty $ (corresponding to consistent model selection).
Suppose $\theta _{n}\in {\mathbb{R}}$ satisfies $\theta _{n}/\eta
_{n}\rightarrow \zeta \in {\mathbb{R}}\cup \{-\infty ,\infty \}$. Then

\begin{enumerate}
\item $\left\vert \zeta \right\vert <1$ implies $\lim_{n\rightarrow \infty
}P_{n,\theta _{n}}(\hat{\theta}=0)=1$;

\item $|\zeta|=1$ and $n^{1/2}(\eta _{n}-\zeta \theta _{n})\rightarrow r$
for some $r\in \mathbb{R\cup \{-\infty },\mathbb{\infty \}}$, implies $%
\lim_{n\rightarrow \infty }P_{n,\theta _{n}}(\hat{\theta}=0)=\Phi (r)$;

\item $|\zeta|>1$ implies $\lim_{n\rightarrow \infty }P_{n,\theta _{n}}(\hat{%
\theta}=0)=0$.
\end{enumerate}
\end{proposition}

\begin{proof}
The proof of part (i) is immediate from (\ref{model_prob}). To prove part
(ii) we use (\ref{model_prob}) to rewrite $P_{n,\theta _{n}}(\hat{\theta}
=0) $ as 
\begin{eqnarray*}
P_{n,\theta _{n}}(\hat{\theta} =0) &=&\Phi (n^{1/2}\eta _{n}(1-\theta
_{n}/\eta _{n}))-\Phi (n^{1/2}\eta _{n}(-1-\theta _{n}/\eta _{n})).
\end{eqnarray*}%
The first and the third claim follow immediately from this. For the second
claim, assume first that $\zeta=1$. Then $\Phi(n^{1/2}\eta_n(1-\theta_n/%
\eta_n)) = \Phi(n^{1/2}(\eta_n - \zeta \theta_n))$ obviously converges to $%
\Phi (r)$, whereas $\Phi(n^{1/2}\eta_n(-1-\theta_n/\eta_n))$ converges to
zero. The case $\zeta = -1$ is handled similarly.
\end{proof}

Proposition~\ref{selection_prob} in fact completely describes the
large-sample behavior of the model selection probability without \textit{any}
conditions on the parameter $\theta $, in the sense that all possible
accumulation points of the model selection probability along \textit{%
arbitrary} sequences of $\theta _{n}$ can be obtained in the following
manner: Just apply the result to subsequences and note that, by compactness
of $\mathbb{R\cup \{-\infty },\mathbb{\infty \}}$, we can select from each
subsequence a further subsequence such that all relevant quantities such as $%
n^{1/2}\theta _{n}$, $\theta _{n}/\eta _{n}$, $n^{1/2}(\eta _{n}-\theta
_{n}) $, or $n^{1/2}(\eta _{n}+\theta _{n})$ converge in $\mathbb{R\cup
\{-\infty },\mathbb{\infty \}}$ along this further subsequence.

In the conservative model selection case we see from Proposition \ref%
{selection_prob} that the usual local alternative parameter sequences
describe the asymptotic behavior. In particular, if $\theta _{n}$ is local
to $\theta =0$ in the sense that $\theta _{n}=\nu /n^{1/2}$, the local
alternatives parameter $\nu $ governs the limiting model selection
probability. Deviations of $\theta _{n}$ from $\theta =0$ of order $%
1/n^{1/2} $ are detected with positive probability asymptotically and
deviations of larger order are detected with probability one asymptotically
in this case. In the consistent model selection case, however, a different
picture emerges. Here, Proposition \ref{selection_prob} shows that local
deviations of $\theta _{n}$ from $\theta =0$ that are of the order $%
1/n^{1/2} $ are not detected by the model selection procedures at all!%
\footnote{%
For such deviations this also immediately follows from a contiguity argument.%
} In fact, even larger deviations of $\theta $ from zero go asymptotically
unnoticed by the model selection procedure, namely as long as $\theta
_{n}/\eta _{n}\rightarrow \zeta $, $\left\vert \zeta \right\vert <1$. [Note
that these larger deviations would be picked up by a \emph{conservative}
procedure with probability one asymptotically.] This unpleasant consequence
of model selection consistency has a number of repercussions as we shall see
later on. For a more detailed discussion of these phenomena in the context
of post-model-selection estimators see Leeb and P\"{o}tscher (2005).

The speed of convergence of the model selection probability to its limit in
part (i) of the proposition is governed by the slower of the convergence
speeds of $n^{1/2}\eta _{n}$ and $n^{1/2}\theta _{n}$. In part (ii) it is
exponential in $n^{1/2}\eta _{n}$ in cases 1 and 3, and is governed by the
convergence speed of $n^{1/2}\eta _{n}$ and $n^{1/2}(\eta _{n}-\zeta \theta
_{n})$ in case 2.

\section{Consistency, uniform consistency, and uniform convergence rate of $%
\hat{\protect\theta}_{H}$, $\hat{\protect\theta}_{S}$, and $\hat{\protect%
\theta}_{SCAD}$\label{cons&ucons}}

It is easy to see that the basic condition $\eta _{n}\rightarrow 0$
discussed in the preceding section is in fact also equivalent to consistency
of $\hat{\theta}_{H}$ for $\theta $, i.e., to 
\begin{equation}
\lim_{n\rightarrow \infty }P_{n,\theta }\left( \left\vert \hat{\theta}%
_{H}-\theta \right\vert >\varepsilon \right) \;\;=\;\;0\quad \text{for every 
}\varepsilon >0\text{ and every }\theta \in \mathbb{R}.  \notag  \label{cons}
\end{equation}%
The same is also true for $\hat{\theta}_{S}$ and $\hat{\theta}_{SCAD}$, as
is elementary to verify. [At least the sufficiency parts are well-known, see
P\"{o}tscher (1991) for hard-thresholding, Knight and Fu (2000) for
soft-thresholding\footnote{%
Knight and Fu (2000) consider the LASSO-estimator in a linear regression
model without an intercept, hence their result does not directly apply to
the case considered here.}, and Fan and Li (2001) for SCAD.] In fact, under
this basic condition on $\eta _{n}$, the estimators are even uniformly
consistent with a certain rate as we show next:

\begin{theorem}
\label{unif_cons}Assume $\eta _{n}\rightarrow 0$. Let $\hat{\theta}$ stand
for either $\hat{\theta}_{H}$, $\hat{\theta}_{S}$, or $\hat{\theta}_{SCAD}$.
Then $\hat{\theta}$ is uniformly consistent, i.e., 
\begin{equation}
\lim_{n\rightarrow \infty }\sup_{\theta \in \mathbb{R}}P_{n,\theta }\left(
\left\vert \hat{\theta}-\theta \right\vert >\varepsilon \right)
\;\;=\;\;0\quad \text{for every }\varepsilon >0.  \notag  \label{ucons}
\end{equation}%
In fact, the supremum in the above expression converges to zero
exponentially fast for every $\varepsilon >0$. Furthermore, set $a_{n}=\min
\{n^{1/2},\eta _{n}^{-1}\}$. Then for every $\varepsilon >0$ there exists a
(nonnegative) real number $M$ such that%
\begin{equation}
\sup_{n\in {\mathbb{N}}}\sup_{\theta \in \mathbb{R}}P_{n,\theta }\left(
a_{n}\left\vert \hat{\theta}-\theta \right\vert >M\right)
\;\;<\;\;\varepsilon  \notag  \label{unif_cons_rate}
\end{equation}%
holds. In particular, $\hat{\theta}$ is uniformly $\min \{n^{1/2},\eta
_{n}^{-1}\}$-consistent.
\end{theorem}

\begin{proof}
We begin with proving uniform consistency of $\hat{\theta}=\hat{\theta}_{H}$%
. Observe that $\sup_{\theta \in \mathbb{R}}P_{n,\theta }(|\hat{\theta}%
_{H}-\theta |>\varepsilon )$ can be written as 
\begin{equation}
\begin{split}
& \sup_{\theta \in \mathbb{R}}P_{n,\theta }\Big(\Big\vert(\bar{y}-\theta )%
\boldsymbol{1}(\left\vert \bar{y}\right\vert >\eta _{n})-\theta \boldsymbol{1%
}(\left\vert \bar{y}\right\vert \leq \eta _{n})\Big\vert>\varepsilon \Big) \\
& \leq \quad \sup_{\theta \in \mathbb{R}}P_{n,\theta }(\left\vert \bar{y}%
-\theta \right\vert >\varepsilon /2,\;|\bar{y}|>\eta
_{n})\;\;+\;\;\sup_{\theta \in \mathbb{R}}P_{n,\theta }(\left\vert \theta
\right\vert >\varepsilon /2,\text{ }\left\vert \bar{y}\right\vert \leq \eta
_{n}) \\
& \leq \quad \Pr (\left\vert Z\right\vert >n^{1/2}\varepsilon
/2)\;\;+\;\;\sup_{\left\vert \theta \right\vert >\varepsilon /2}P_{n,\theta
}(\left\vert \bar{y}\right\vert \leq \eta _{n}),
\end{split}
\notag
\end{equation}%
where $Z$ is standard normally distributed. Now the first term on the far
r.h.s. in the above display obviously converges to zero exponentially fast
as $n\rightarrow \infty $. In the second term on the far right, the
probability gets large as $|\theta |$ gets close to $\varepsilon /2$.
Therefore, the second term on the far r.h.s. equals 
\begin{equation*}
\Pr \left( \left\vert Z+n^{1/2}\varepsilon /2\right\vert \leq n^{1/2}\eta
_{n}\right) =\Phi (n^{1/2}(-\varepsilon /2+\eta _{n}))-\Phi
(n^{1/2}(-\varepsilon /2-\eta _{n}))
\end{equation*}%
and also goes to zero exponentially fast because $\eta _{n}\rightarrow 0$.

Next, for the soft-thresholding estimator, observe that we have the relation%
\begin{equation}
\hat{\theta}_{S}=\hat{\theta}_{H}-\limfunc{sign}(\hat{\theta}_{H})\eta _{n}.
\label{hard-soft}
\end{equation}%
Consequently, 
\begin{equation*}
\sup_{\theta \in \mathbb{R}}P_{n,\theta }\left( \left\vert \hat{\theta}_{H}-%
\hat{\theta}_{S}\right\vert >\varepsilon \right) \leq \sup_{\theta \in 
\mathbb{R}}P_{n,\theta }(\eta _{n}>\varepsilon )=\boldsymbol{1}(\eta
_{n}>\varepsilon ),
\end{equation*}%
which equals zero for sufficiently large $n$. Hence, the results established
so far for $\hat{\theta}_{H}$ carry over to $\hat{\theta}_{S}$.

For the SCAD estimator observe that it is `sandwiched' between the other two
in the sense that 
\begin{equation}
\hat{\theta}_{S}\leq \hat{\theta}_{SCAD}\leq \hat{\theta}_{H}
\label{sandwich}
\end{equation}%
holds if $\hat{\theta}_{S}\geq 0$, and that the order is reversed if $\hat{%
\theta}_{S}\leq 0$. This entails the corresponding result for the SCAD\
estimator.

We next prove uniform $a_{n}$-consistency of $\hat{\theta}_{H}$: Repeating
the arguments from the beginning of the proof with $M/a_{n}$ replacing $%
\varepsilon $, we see that $\sup_{\theta \in \mathbb{R}}P_{n,\theta }(a_{n}|%
\hat{\theta}_{H}-\theta |>M)$ is bounded from above by 
\begin{equation*}
\Pr \left( \big\vert Z\big\vert>n^{1/2}M/(2a_{n})\right) \quad +\quad \Pr
\left( \big\vert Z+n^{1/2}M/(2a_{n})\big\vert\leq n^{1/2}\eta _{n}\right) .
\end{equation*}%
Because $n^{1/2}/a_{n}\geq 1$, the first term on the right-hand side of the
above expression is not larger than $\Pr (|Z|>M/2)$. The second term equals%
\begin{equation}
\begin{split}
& \Phi \left( -n^{1/2}M/(2a_{n})+n^{1/2}\eta _{n}\right) -\Phi \left(
-n^{1/2}M/(2a_{n})-n^{1/2}\eta _{n}\right) \\
& =\quad \Phi \left( (n^{1/2}/a_{n})(-M/2+a_{n}\eta _{n})\right) -\Phi
\left( (n^{1/2}/a_{n})(-M/2-a_{n}\eta _{n})\right) .
\end{split}
\notag
\end{equation}%
Note that $n^{1/2}/a_{n}\geq 1$ and $a_{n}\eta _{n}\leq 1$. For $M>2$, the
expression in the above display is therefore not larger than $\Phi (-M/2+1)$%
. Uniform $a_{n}$-consistency of $\hat{\theta}_{H}$ follows from this. The
proof for $\hat{\theta}_{S}$ and $\hat{\theta}_{SCAD}$ is then similar as
before.
\end{proof}

For the case where the estimators $\hat{\theta}_{H}$, $\hat{\theta}_{S}$,
and $\hat{\theta}_{SCAD}$ are tuned to perform conservative model selection,
the preceding theorem shows that these estimators are uniformly $n^{1/2}$%
-consistent. In contrast, in case the estimators are tuned to perform
consistent model selection, the theorem only guarantees uniform $\eta
_{n}^{-1}$-consistency; that the estimators do actually not converge faster
than $\eta _{n}$ in a uniform sense in this case will be shown in Section %
\ref{sec_consist}.

\begin{remark}
\normalfont\label{r1}Let $\hat{\theta}$ denote any one of the estimators $%
\hat{\theta}_{H}$, $\hat{\theta}_{S}$, or $\hat{\theta}_{SCAD}$. In case $%
n^{1/2}\eta _{n}\rightarrow e=0$ it is easy to see that $\hat{\theta}$ is
uniformly asymptotically equivalent to $\hat{\theta}_{U}=\bar{y}$ in the
sense that $\lim_{n\rightarrow \infty }\sup_{\theta \in \mathbb{R}%
}P_{n,\theta }\left( n^{1/2}\left\vert \hat{\theta}-\bar{y}\right\vert
>\varepsilon \right) =0$ for every $\varepsilon >0$. [For $\hat{\theta}=\hat{%
\theta}_{H}$, this follows easily from Proposition \ref{selection_prob}, for 
$\hat{\theta}=\hat{\theta}_{S}$ it follows then from (\ref{hard-soft}), and
for $\hat{\theta}=\hat{\theta}_{SCAD}$ from (\ref{sandwich}).]
\end{remark}

\section{The distributions of $\hat{\protect\theta}_{H}$, $\hat{\protect%
\theta}_{S}$, and $\hat{\protect\theta}_{SCAD}$}

\subsection{Finite-sample distributions\label{finite}}

For purpose of comparison we note the obvious fact that the distribution of
the unrestricted maximum likelihood estimator $\hat{\theta}_{U}=\bar{y}$\
(corresponding to model $M_{U}$) as well as the distribution of the
restricted maximum likelihood estimator $\hat{\theta}_{R}\equiv 0$
(corresponding to model $M_{R}$) are normal; more precisely, $n^{1/2}(\hat{%
\theta}_{U}-\theta )$ is $N(0,1)$-distributed and $n^{1/2}(\hat{\theta}%
_{R}-\theta )$ is $N(-n^{1/2}\theta ,0)$-distributed, where the singular
normal distribution is to be interpreted as pointmass at $-n^{1/2}\theta $.
For the hard-thresholding estimator, the finite-sample distribution $%
F_{H,n,\theta }$ of $n^{1/2}(\hat{\theta}_{H}-\theta )$ is of the form%
\begin{equation}
\begin{split}
dF_{H,n,\theta }(x)\quad =\quad & \left\{ \Phi (n^{1/2}(-\theta +\eta
_{n}))-\Phi (n^{1/2}(-\theta -\eta _{n}))\right\} d\delta _{-n^{1/2}\theta
}(x) \\
& \quad +\quad \phi (x)\;\boldsymbol{1}\left( \left\vert x+n^{1/2}\theta
\right\vert >n^{1/2}\eta _{n}\right) dx,
\end{split}
\label{distri_H}
\end{equation}%
where $\delta _{z}$ denotes pointmass at $z$ and $\phi $ denotes the
standard normal density. Relation (\ref{distri_H}) is most easily obtained
by writing $P_{n,\theta }(n^{1/2}(\hat{\theta}_{H}-\theta )\leq x)$ as the
sum of $P_{n,\theta }(n^{1/2}(\hat{\theta}_{H}-\theta )\leq x,\hat{\theta}%
_{H}=0)$ and $P_{n,\theta }(n^{1/2}(\hat{\theta}_{H}-\theta )\leq x,\hat{%
\theta}_{H}\neq 0)$. This also shows that the two terms in (\ref{distri_H})
correspond to the distribution of $n^{1/2}(\hat{\theta}_{H}-\theta )$
conditional on the events $\{\hat{\theta}_{H}=0\}$ and $\{\hat{\theta}%
_{H}\neq 0\}$, respectively, multiplied by the probability of the respective
events. Relation (\ref{distri_H}) also follows as a special case of Leeb and
P\"{o}tscher (2003), which provides the finite-sample as well as the
asymptotic distributions of a general class of post-model-selection
estimators. We recognize that the distribution of the hard-thresholding
estimator is a mixture of two components: The first one is a singular normal
distribution (i.e., pointmass) and coincides with the distribution of the
restricted maximum likelihood estimator. The second one is absolutely
continuous and represents an `excised' version of the normal distribution of
the unrestricted maximum likelihood estimator. Note that the absolutely
continuous part in (\ref{distri_H}) is bimodal and hence is distinctly
non-normal. The shape of the distribution of $n^{1/2}(\hat{\theta}%
_{H}-\theta )$ is exemplified in Figure~1.

\begin{center}
\begin{tabular}{c}
\pgfuseimage{figure1}%
\end{tabular}
\end{center}

\begin{quote}
\normalfont Figure~1: Distribution of $n^{1/2}(\hat{\theta}_{H}-\theta )$
for $n=40$, $\theta =0.16$, and $\eta _{n}=0.05$. The density of the
absolutely continuous part is shown by the solid curve, which is
discontinuous at $x=n^{1/2}(-\theta -\eta _{n})$ and $x=n^{1/2}(-\theta
+\eta _{n})$. [For better readability, the left- and right-hand limits at
discontinuity points are joined by line segments.] The vertical dotted line
indicates the location of the point-mass at $-n^{1/2}\theta $; the weight of
the point-mass, i.e., the multiplier of $d\delta _{-n^{1/2}\theta }(x)$ in (%
\ref{distri_H}), equals $0.15$. For other values of the constants involved
here, a similar picture is obtained.
\end{quote}

The finite-sample distribution $F_{S,n,\theta }$ of $n^{1/2}(\hat{\theta}%
_{S}-\theta )$ is given by%
\begin{equation}
\begin{split}
dF_{S,n,\theta }(x)\quad =\quad & \left\{ \Phi (n^{1/2}(-\theta +\eta
_{n}))-\Phi (n^{1/2}(-\theta -\eta _{n}))\right\} d\delta _{-n^{1/2}\theta
}(x) \\
& \;+\;\phi (x-n^{1/2}\eta _{n})\;\boldsymbol{1}(x+n^{1/2}\theta <0)\;dx \\
& \;+\;\phi (x+n^{1/2}\eta _{n})\;\boldsymbol{1}(x+n^{1/2}\theta >0)\;dx.
\end{split}
\label{distri_S}
\end{equation}%
For later use we note that this implies%
\begin{equation}
F_{S,n,\theta }(x)\quad =\quad \Phi (x+n^{1/2}\eta _{n}))\boldsymbol{1}%
(x\geq -n^{1/2}\theta )+\Phi (x-n^{1/2}\eta _{n}))\boldsymbol{1}%
(x<-n^{1/2}\theta ).  \label{distri_S_1}
\end{equation}%
Relation (\ref{distri_S}) is obtained from a derivation similar to that of (%
\ref{distri_H}), namely by representing $P_{n,\theta }(n^{1/2}(\hat{\theta}%
_{S}-\theta )\leq x)$ as the sum of $P_{n,\theta }(n^{1/2}(\hat{\theta}%
_{S}-\theta )\leq x,\hat{\theta}_{S}=0)$, $P_{n,\theta }(n^{1/2}(\hat{\theta}%
_{S}-\theta )\leq x,\hat{\theta}_{S}>0)$, and $P_{n,\theta }(n^{1/2}(\hat{%
\theta}_{S}-\theta )\leq x,\hat{\theta}_{S}<0)$. Similar to before, the
three terms in (\ref{distri_S}) correspond to the distributions of $n^{1/2}(%
\hat{\theta}_{S}-\theta )$ conditional on the events $\{\hat{\theta}_{S}=0\}$%
, $\{\hat{\theta}_{S}>0\}$, and $\{\hat{\theta}_{S}<0\}$, respectively,
multiplied by the respective probabilities of these events. The distribution
in (\ref{distri_S}) is again a mixture of a singular normal distribution and
of an absolutely continuous part, which is now the sum of two normal
densities, each with a truncated tail. Figure~2 exemplifies a typical shape
of this distribution.

\begin{center}
\begin{tabular}{c}
\pgfuseimage{figure2}%
\end{tabular}
\end{center}

\begin{quote}
\normalfont
Figure~2: Distribution of $n^{1/2}(\hat{\theta}_S - \theta)$. The choice of
constants and the interpretation of the image is the same as in Figure~1.
\end{quote}

The finite-sample distribution of the SCAD-estimator is obtained in a
similar vein: Decomposing the probability $P_{n,\theta }(n^{1/2}(\hat{\theta}%
_{SCAD}-\theta )\leq x)$ into a sum of seven terms by decomposing the
relevant event into its intersection with the events $\{\left\vert \bar{y}%
\right\vert \leq \eta _{n}\}$, $\{\eta _{n}<\bar{y}\leq 2\eta _{n}\}$, $%
\{2\eta _{n}<\bar{y}\leq a\eta _{n}\}$, $\{a\eta _{n}<\bar{y}\}$, $\{-2\eta
_{n}\leq \bar{y}<-\eta _{n}\}$, $\{-a\eta _{n}\leq \bar{y}<-2\eta _{n}\}$,
and $\{\bar{y}<-a\eta _{n}\}$, shows that the distribution $F_{SCAD,n,\theta
}$ of $n^{1/2}(\hat{\theta}_{SCAD}-\theta )$ is of the form%
\begin{equation}
\begin{split}
dF_{SCAD,n,\theta }& (x)\;=\;\Big\{\Phi (n^{1/2}(-\theta +\eta _{n}))-\Phi
(n^{1/2}(-\theta -\eta _{n}))\Big\}\;d\delta _{-n^{1/2}\theta }(x) \\
& +\;\Big\{f_{1}(x)+f_{2}(x)+f_{3}(x)+f_{-1}(x)+f_{-2}(x)+f_{-3}(x)\Big\}%
\;dx,
\end{split}
\label{distri_SCAD}
\end{equation}%
where 
\begin{eqnarray*}
f_{1}(x) &=&\phi \left( x+n^{1/2}\eta _{n}\right) \;\boldsymbol{1}\left(
0<x+n^{1/2}\theta \leq n^{1/2}\eta _{n}\right) , \\
f_{2}(x) &=&\frac{a-2}{a-1}\phi \left( \left\{ (a-2)x-n^{1/2}\theta
+an^{1/2}\eta _{n}\right\} /(a-1)\right) \times \\
\raisetag{2cm} &&\quad \boldsymbol{1}\left( n^{1/2}\eta _{n}<x+n^{1/2}\theta
\leq an^{1/2}\eta _{n}\right) , \\
f_{3}(x) &=&\phi \left( x\right) \;\boldsymbol{1}\left( x+n^{1/2}\theta
>n^{1/2}a\eta _{n}\right) ,
\end{eqnarray*}%
and where $f_{-1}(x)$, $f_{-2}(x)$, and $f_{-3}(x)$ are defined as $f_{1}(x)$%
, $f_{2}(x)$, and $f_{3}(x)$, respectively, but with $-x$ replacing $x$ and
with $-\theta $ replacing $\theta $ in the formulae. Like in the case of the
other estimators, the distribution of the SCAD-estimator is a mixture of a
singular normal distribution and an absolutely continuous part, the latter
being more complicated here as it is the sum of six pieces, each obtained
from normal distributions by truncation or excision. As shown in Figure~3,
the absolutely continuous part of $F_{SCAD,n,\theta }$ can be multimodal.

\begin{center}
\begin{tabular}{c}
\pgfuseimage{figure3}%
\end{tabular}
\end{center}

\begin{quote}
\normalfont Figure~3: Distribution of $n^{1/2}(\hat{\theta}_{SCAD}-\theta )$%
. The tuning-parameter $a$ is chosen as $a=3.7$ here, cf. Fan and Li (2001);
the choice of the other constants and the interpretation of the image is the
same as in Figure~1. The graph for the SCAD estimator coincides with that
for the soft-thresholding estimator inside a neighborhood of the location of
the atomic part at $-n^{1/2}\theta $ (vertical dotted line), and with that
for the hard-thresholding estimator outside of a (larger) neighborhood of $%
-n^{1/2}\theta $. The area between these two regions corresponds to the dips
shown in the figure.
\end{quote}

In summary, we see that the finite-sample distributions of the estimators $%
\hat{\theta}_{H}$, $\hat{\theta}_{S}$, and $\hat{\theta}_{SCAD}$ are
typically highly non-normal and can be multimodal. As a point of interest,
we also note that the computations leading to the above formulae also
deliver the conditional finite-sample distributions of the estimators $\hat{%
\theta}_{H}$, $\hat{\theta}_{S}$, and $\hat{\theta}_{SCAD}$, respectively,
conditional on selecting model $M_{R}$ or $M_{U}$. In particular, we note
that the conditional distribution of each of these estimators, conditional
on having selected the restricted model $M_{R}$, coincides with the
distribution of the restricted maximum likelihood estimator $\hat{\theta}%
_{R} $; in contrast, conditional on selecting the unrestricted model $M_{U}$%
, the conditional distribution is not identical to the distribution of the
unrestricted maximum likelihood estimator $\hat{\theta}_{U}$, but is more
complicated. This phenomenon applies also to large classes of
post-model-selection estimators; see P\"{o}tscher (1991) and Leeb and P\"{o}%
tscher (2003) for more discussion.

\subsection{Asymptotic distributions\label{asydistribs}}

We next obtain the asymptotic distributions of $\hat{\theta}_{H}$, $\hat{%
\theta}_{S}$, and $\hat{\theta}_{SCAD}$. We present the asymptotic
distributional results under general `moving parameter' asymptotics, where
the true parameter $\theta _{n}$ can depend on sample size, because
considering only fixed-parameter asymptotics may paint a quite misleading
picture of the behavior of the estimators (cf. Leeb and P\"{o}tscher (2003,
2005)). In fact, the results given below amount to a complete description of
all possible accumulation points of the finite-sample distributions of the
estimators in question, cf. Remarks \ref{rxx} and \ref{rxxx}. Not
surprisingly, the results in the conservative model selection case are
different from the ones in the consistent model selection case.

\subsubsection{Conservative case\label{sec_conserv}}

Here we characterize the large-sample behavior of the distributions of $\hat{%
\theta}_{H}$, $\hat{\theta}_{S}$, and $\hat{\theta}_{SCAD}$ for the case
where these estimators are tuned to perform conservative model selection.

\begin{theorem}
\label{H_conserv} Consider the hard-thresholding estimator with $\eta
_{n}\rightarrow 0$ and $n^{1/2}\eta _{n}\rightarrow e$, $0\leq e<\infty $.
Suppose the true parameter $\theta _{n}\in \mathbb{R}$ satisfies $%
n^{1/2}\theta _{n}\rightarrow \nu \in \mathbb{R}\cup \{-\infty ,\infty \}$.
Then $F_{H,n,\theta _{n}}$ converges weakly to the distribution given by%
\begin{equation}
\left\{ \Phi (-\nu +e)-\Phi (-\nu -e)\right\} \;d\delta _{-\nu }(x)
\;\;+\;\;\phi (x)\boldsymbol{1}(\left\vert x+\nu \right\vert >e)\;dx.
\label{asydistr_H_conserv}
\end{equation}%
[Note that (\ref{asydistr_H_conserv}) reduces to a standard normal
distribution in case $\left\vert \nu \right\vert =\infty$ or $e=0$.]
\end{theorem}

\begin{proof}
\footnote{%
Theorem \ref{H_conserv} is essentially a special case of results obtained in
Leeb and P\"{o}tscher (2003) for a more general class of
post-model-selection estimators. The proof of this result is included here
because of its brevity and illustrative value.} Recall that the
finite-sample distribution is given in (\ref{distri_H}). Convergence of the
weights $\Phi (n^{1/2}(-\theta +\eta _{n}))-\Phi (n^{1/2}(-\theta -\eta
_{n}))$ to $\Phi (-\nu +e)-\Phi (-\nu -e)$ is obvious (cf. proof of
Proposition 1). Hence, the atomic part of $F_{H,n,\theta _{n}}$ converges
weakly to the atomic part of (\ref{asydistr_H_conserv}) if $\left\vert \nu
\right\vert <\infty $ and $e>0$; if $\left\vert \nu \right\vert =\infty $ or
if $e=0$, the total mass of the atomic part converges to zero. The density
of the absolutely continuous part of $F_{H,n,\theta _{n}}$ is easily seen to
converge Lebesgue almost everywhere (in fact everywhere on $\mathbb{R}$
except possibly at $x=-\nu \pm e$) to the density of the absolutely
continuous part of (\ref{asydistr_H_conserv}). Also the total mass of the
absolutely continuous part is seen to converge to the total mass of the
absolutely continuous part of (\ref{asydistr_H_conserv}). By an application
of Scheff\'{e}'s Lemma, the densities converge in absolute mean, and hence
the absolutely continuous part converges in the total variation sense.
\end{proof}

The fixed-parameter asymptotic distribution is obtained from Theorem~\ref%
{H_conserv} by letting $\theta _{n}\equiv \theta $: If $\theta =0$, the
pointwise asymptotic distribution of the hard-thresholding estimator is seen
to be 
\begin{equation*}
\left\{ \Phi (e)-\Phi (-e)\right\} \;d\delta _{0}(x)\;\;+\;\;\phi (x)\;%
\boldsymbol{1}(\left\vert x\right\vert >e)\;dx,
\end{equation*}%
which coincides with the finite-sample distribution (\ref{distri_H}) in this
case except for replacing $n^{1/2}\eta _{n}$ by its limiting value $e$.
However, if $\theta \neq 0$, the pointwise asymptotic distribution is always
standard normal, which clearly misrepresents the actual distribution (\ref%
{distri_H}). This disagreement is most pronounced in the statistically
interesting case where $\theta $ is close to, but not equal to, zero (e.g., $%
\theta \sim n^{-1/2}$). In contrast, the distribution (\ref%
{asydistr_H_conserv}) much better captures the behavior of the finite-sample
distribution also in this case because (\ref{asydistr_H_conserv}) coincides
with (\ref{distri_H}) except for the fact that $n^{1/2}\eta _{n}$ and $%
n^{1/2}\theta _{n}$ have settled down to their limiting values.

\begin{theorem}
\label{S_conserv} Consider the soft-thresholding estimator with $\eta
_{n}\rightarrow 0$ and $n^{1/2}\eta _{n}\rightarrow e$, $0\leq e<\infty $.
Suppose the true parameter $\theta _{n}\in \mathbb{R}$ satisfies $%
n^{1/2}\theta _{n}\rightarrow \nu \in \mathbb{R}\cup \{-\infty ,\infty \}$.
Then $F_{S,n,\theta _{n}}$ converges weakly to the distribution given by%
\begin{equation}
\begin{split}
& \left\{ \Phi (-\nu +e)-\Phi (-\nu -e)\right\} \;d\delta _{-\nu }(x) \\
& \;+\;\;\left\{ \phi (x+e)\boldsymbol{1}(x>-\nu )\;+\;\phi (x-e)\boldsymbol{%
1}(x<-\nu )\right\} \;dx.
\end{split}
\label{asydistr_S_conserv}
\end{equation}%
[Note that (\ref{asydistr_S_conserv}) reduces to a $N(-\limfunc{sign}(\nu
)e,1)$-distribution in case $\left\vert \nu \right\vert =\infty $ or $e=0$.]
\end{theorem}

The proof is completely analogous to the proof of Theorem \ref{H_conserv}.
Since soft-thresholding arises as a special case of the LASSO-estimator, the
above result is closely related to the results in Knight and Fu (2000).%
\footnote{%
Since Knight and Fu (2000) consider the LASSO-estimator in a linear
regression model without an intercept, their results do not directly apply
to the model considered here. However, their results can easily be modified
to also cover linear regression with an intercept.} Similar to the case of
hard-thresholding, a fixed-parameter asymptotic analysis only partially
reflects the finite-sample behavior of the estimator: In case $\theta =0$,
the pointwise asymptotic distribution is 
\begin{equation*}
\left\{ \Phi (e)-\Phi (-e)\right\} \;d\delta _{0}(x)\;\;+\;\;\left\{ \phi
(x-e)\boldsymbol{1}(x<0)\;+\;\phi (x+e)\boldsymbol{1}(x>0)\right\} \;dx.
\end{equation*}%
However, if $\theta \neq 0$, the pointwise limit distribution is $N(-%
\limfunc{sign}(\theta )e,1)$, which is not in good agreement with the
finite-sample distribution (\ref{distri_S}), especially in the statistically
interesting case where $\theta $ is close to, but not equal to, zero (e.g., $%
\theta \sim n^{-1/2}$). In contrast, (\ref{asydistr_S_conserv}) is in better
agreement with (\ref{distri_S}) also in this case in the sense that (\ref%
{asydistr_S_conserv}) coincides with (\ref{distri_S}), except that $%
n^{1/2}\eta _{n}$ and $n^{1/2}\theta _{n}$ have settled down to their
limiting values.

\begin{theorem}
\label{SCAD_conser} Consider the SCAD estimator with $\eta _{n}\rightarrow 0$
and $n^{1/2}\eta _{n}\rightarrow e$, $0\leq e<\infty $. Suppose the true
parameter $\theta _{n}\in \mathbb{R}$ satisfies $n^{1/2}\theta
_{n}\rightarrow \nu \in \mathbb{R}\cup \{-\infty ,\infty \}$. Then $%
F_{SCAD,n,\theta _{n}}$ converges weakly to the distribution given by%
\begin{equation}
\begin{split}
& \left\{ \Phi (-\nu +e)-\Phi (-\nu -e)\right\} \;d\delta _{-\nu }(x) \\
& +\;\;\Big\{\phi (x+e)\boldsymbol{1}(0<x+\nu \leq e)+\phi (x-e)\boldsymbol{1%
}(-e\leq x+\nu <0) \\
& \qquad +\;\;\frac{a-2}{a-1}\phi (\left\{ (a-2)x-\nu +ae\right\} /(a-1))%
\boldsymbol{1}(e<x+\nu \leq ae) \\
& \qquad +\;\;\frac{a-2}{a-1}\phi (\left\{ (a-2)x-\nu -ae\right\} /(a-1))%
\boldsymbol{1}(-ae\leq x+\nu <-e) \\
& \qquad +\;\;\phi (x)\boldsymbol{1}(\left\vert x+\nu \right\vert >ae)\Big\}%
\;dx.
\end{split}
\label{asydistr_SCAD_conserv}
\end{equation}%
[Note that (\ref{asydistr_SCAD_conserv}) reduces to a standard normal
distribution in case $\left\vert \nu \right\vert =\infty $ or $e=0$.]
\end{theorem}

The proof of Theorem~\ref{SCAD_conser} is again completely analogous to that
of Theorem \ref{H_conserv}. As with the hard- and soft-thresholding
estimators discussed before, a fixed-parameter asymptotic analysis of the
SCAD estimator only partially reflects its finite-sample behavior: In case $%
\theta =0$, the pointwise asymptotic distribution is given by (\ref%
{asydistr_SCAD_conserv}) with $\nu =0$, but in case $\theta \neq 0$ it is
given by $N(0,1)$, which is definitely not in good agreement with the
finite-sample distribution (\ref{distri_SCAD}), especially in the
statistically interesting case where $\theta $ is different from, but close
to, zero, e.g., $\theta \sim n^{-1/2}$. In contrast, (\ref%
{asydistr_SCAD_conserv}) is in much better agreement with (\ref{distri_SCAD}%
) in view of the fact that (\ref{asydistr_SCAD_conserv}) coincides with (\ref%
{distri_SCAD}), except that $n^{1/2}\eta _{n}$ and $n^{1/2}\theta _{n}$ have
settled down to their limiting values.

We note that the mathematical reason for the failure of the pointwise
asymptotic distribution to capture the behavior of the finite-sample
distribution well is that the convergence of the latter to the former is not
uniform in the underlying parameter $\theta $. See Leeb and P\"{o}tscher
(2003, 2005) for more discussion in the context of post-model-selection
estimators.

\begin{remark}
\normalfont\label{rx} If $\left\vert \nu \right\vert =\infty $, or $e=0$, or 
$n^{1/2}\theta _{n}=\nu $ does not depend on $n$, the convergence in the
above three theorems is even in the total variation distance. In the first
two cases this follows because the total mass of the atomic part converges
to zero; in the third case it follows because the location of the pointmass
is independent of $n$.
\end{remark}

\begin{remark}
\normalfont\label{rxx} The above theorems actually completely describe the
limiting behavior of the finite-sample distributions of $\hat{\theta}_{H}$, $%
\hat{\theta}_{S}$, and $\hat{\theta}_{SCAD}$ without \emph{any }condition on
the sequence of parameters $\theta _{n}$. To see this, just apply the
theorems to subsequences and note that by compactness of ${\mathbb{R}}\cup
\{-\infty ,\infty \}$ we can select from every subsequence a further
subsequence such that $n^{1/2}\theta _{n}$ converges in ${\mathbb{R}}\cup
\{-\infty ,\infty \}$ along this further subsequence.
\end{remark}

\subsubsection{Consistent case\label{sec_consist}}

In the case where the estimators $\hat{\theta}_{H}$, $\hat{\theta}_{S}$, and 
$\hat{\theta}_{SCAD}$ are tuned to perform consistent model selection (i.e., 
$\eta _{n}\rightarrow 0$ and $n^{1/2}\eta _{n}\rightarrow \infty $), the 
\emph{fixed-parameter} limiting behavior of the finite-sample distributions
is particularly simple: The finite-sample distribution of the
hard-thresholding estimator converges to the $N(0,0)$-distribution (i.e., to
pointmass at $0$) if $\theta =0$, and to the $N(0,1)$-distribution if $%
\theta \neq 0$; cf. Lemma~1 in P\"{o}tscher (1991). In other words, the
pointwise asymptotic distribution of $n^{1/2}(\hat{\theta}_{H}-\theta )$
coincides with the asymptotic distribution of the restricted maximum
likelihood estimator if $\theta =0$, and coincides with the asymptotic
distribution of the unrestricted maximum likelihood estimator if $\theta
\neq 0$. The hard-thresholding estimator, when tuned in this way, therefore
satisfies what has sometimes been dubbed the `oracle' property in the
literature.\footnote{%
This does not come as a surprise, since post-model-selection estimators
based on a consistent model selection procedure in general satisfy the
`oracle' property as already noted in Lemma 1 of P\"{o}tscher (1991); but
see also the warning issued in the discussion following that lemma.} The
SCAD-estimator with the same tuning is also known to possess the `oracle'
property; cf. Fan and Li (2001). With the same tuning, the soft-thresholding
has a somewhat different pointwise asymptotic behavior which is discussed
later.

The `oracle' property of the hard-thresholding estimator and the
SCAD-estimator implies in particular that both estimators are $n^{1/2}$%
-consistent. In Theorem \ref{unif_cons}, however, we have -- in contrast to
the conservative model selection case -- only been able to establish uniform 
$\eta _{n}^{-1}$-consistency and not uniform $n^{1/2}$-consistency. This
begs the question whether Theorem \ref{unif_cons} is just not sharp enough
or whether the estimators actually are not uniformly $n^{1/2}$-consistent.
It furthermore raises the question of the behavior of the finite-sample
distributions of $n^{1/2}(\hat{\theta}_{H}-\theta )$, $n^{1/2}(\hat{\theta}%
_{S}-\theta )$, and $n^{1/2}(\hat{\theta}_{SCAD}-\theta )$ in a `uniform'
asymptotic framework. The three results that follow answer this by
determining the limits of the finite-sample distributions of $\hat{\theta}%
_{H}$, $\hat{\theta}_{S}$, and $\hat{\theta}_{SCAD}$ under general `moving
parameter' asymptotics when the estimators are tuned to perform consistent
model selection.

\begin{theorem}
\label{H_consist} Consider the hard-thresholding estimator with $\eta
_{n}\rightarrow 0$ and $n^{1/2}\eta _{n}\rightarrow \infty $. Assume that $%
\theta _{n}/\eta _{n}\rightarrow \zeta $ for some $\zeta \in {\mathbb{R}}%
\cup \{-\infty ,\infty \}$ and that $n^{1/2}\theta _{n}\rightarrow \nu $ for
some $\nu \in {\mathbb{R}}\cup \{-\infty ,\infty \}$. [Note that in case $%
\zeta \neq 0$ the convergence of $n^{1/2}\theta _{n}$ already follows from
that of $\theta _{n}/\eta _{n}$, and $\nu $ is then given by $\nu =\limfunc{%
sign}(\zeta )\infty $.]

\begin{enumerate}
\item If $|\zeta |<1$, then $F_{H,n,\theta _{n}}$ approaches pointmass at $%
-\nu $. In case $|\nu |<\infty $, this means that $F_{H,n,\theta _{n}}$
converges weakly to pointmass at $-\nu $; in case $|\nu |=\infty $, this
means that the total mass of $F_{H,n,\theta _{n}}$ escapes to $-\nu $, in
the sense that $F_{H,n,\theta _{n}}(x)\rightarrow 0$ for every $x\in \mathbb{%
R}$ if $-\nu =\infty $, and $F_{H,n,\theta _{n}}(x)\rightarrow 1$ for every $%
x\in \mathbb{R}$ if $-\nu =-\infty $.

\item If $|\zeta |=1$ and $n^{1/2}(\eta _{n}-\zeta \theta _{n})\rightarrow r$
for some $r\in \mathbb{R\cup \{-\infty },\mathbb{\infty \}}$, then $%
F_{H,n,\theta _{n}}(x)$ converges to%
\begin{equation*}
\Phi (r)\boldsymbol{1}(\zeta =1)\;+\;\int_{-\infty }^{x}\phi (u)\boldsymbol{1%
}(\zeta u>r)du
\end{equation*}%
for every $x\in \mathbb{R}$. This limit corresponds to pointmass at $-\nu =%
\limfunc{sign}(-\zeta )\infty $ if $r=\infty $, and otherwise represents a
convex combination of pointmass at $-\nu =\limfunc{sign}(-\zeta )\infty $
and an absolutely continuous distribution whose density, a kind of truncated
standard normal, is given by $(1-\Phi (r))^{-1}$ times the integrand in the
above formula; the weights in that convex combination are given by $\Phi (r)$
and $(1-\Phi (r))$, respectively. [The weight of the absolutely continuous
component equals one in case $r=-\infty $; in this case, convergence is in
fact in total variation distance.]

\item If $1<\left\vert \zeta \right\vert \leq \infty $, then $F_{H,n,\theta
_{n}}$ converges weakly to $\Phi $, the standard normal cdf. [In fact,
convergence is in total variation distance.]
\end{enumerate}
\end{theorem}

\begin{proof}
Proposition \ref{selection_prob} shows that the total mass of the atomic
part of $F_{H,n,\theta _{n}}$ converges to one under the conditions of
part~1. Because the atomic part is located at $-n^{1/2}\theta_n$ in view of (%
\ref{distri_H}), part~1 follows immediately.

For part 2, assume first that $\zeta =1$. Proposition \ref{selection_prob}
shows that the total mass of the atomic part of $F_{H,n,\theta _{n}}$
converges to $\Phi (r)$. Furthermore, $n^{1/2}\theta _{n}\rightarrow \infty $
certainly holds, which implies that the atomic part escapes to $-\infty $.
If $r=\infty $, we are hence done. Suppose now that $r<\infty $. In (\ref%
{distri_H}), the boundaries of the `excision interval' of the absolutely
continuous part of $F_{H,n,\theta _{n}}$, i.e., $-n^{1/2}(\eta _{n}+\theta
_{n})=-n^{1/2}\eta _{n}(1+\theta _{n}/\eta _{n})$ and $n^{1/2}(\eta
_{n}-\theta _{n})$ then converge to $-\infty $ and $r$, respectively. This
shows that 
\begin{equation*}
\phi (x)\;\boldsymbol{1}(|x+n^{1/2}\theta _{n}|>n^{1/2}\eta _{n})\quad
\rightarrow \quad \phi (x)\;\boldsymbol{1}(x>r)
\end{equation*}%
for Lebesgue almost every $x\in \mathbb{R}$. The Dominated Convergence
Theorem then shows that the convergence in the above display also holds in
absolute mean. This completes the proof of part 2 in case $\zeta =1$. The
case where $\zeta =-1$ is treated similarly.

Under the conditions of part 3, Proposition \ref{selection_prob} shows that
the total mass of the absolutely continuous part converges to one.
Furthermore, the boundaries of the `excision interval' in (\ref{distri_H}),
i.e., $-n^{1/2}(\eta _{n}+\theta _{n})=-n^{1/2}\eta _{n}(1+\theta _{n}/\eta
_{n})$ and $n^{1/2}(\eta _{n}-\theta _{n})=n^{1/2}\eta _{n}(1-\theta
_{n}/\eta _{n})$, diverge either both to $\infty $ or both to $-\infty $,
because $|\zeta |>1$. This implies that%
\begin{equation*}
\phi (x)\;\boldsymbol{1}(|x+n^{1/2}\theta _{n}|>n^{1/2}\eta _{n})\quad
\rightarrow \quad \phi (x)
\end{equation*}%
for every $x\in \mathbb{R}$. Together with the Dominated Convergence Theorem
this completes the proof.
\end{proof}

The fixed-parameter asymptotic behavior of the hard-thresholding estimator
discussed earlier, including the `oracle' property, can clearly be recovered
from the above theorem by setting $\theta _{n}\equiv \theta $. However, the
theorem shows that the asymptotic behavior of the hard-thresholding
estimator is more complicated than what the `oracle' property predicts. In
particular, the theorem shows that the hard-thresholding estimator is not
uniformly $n^{1/2}$-consistent as the sequence of finite-sample
distributions is not stochastically bounded in all cases. [In that sense
scaling by $n^{1/2}$ does not appear to be the natural thing to do, see the
discussion below as well as the appendix.] Furthermore, as shown by (\ref%
{distri_H}), the finite-sample distribution is highly non-normal, whereas
the \textit{pointwise} asymptotic distribution is always normal and thus can
not capture essential features of the finite-sample distribution. In
contrast, the asymptotic distribution given in Theorem \ref{H_consist} is
also non-normal in some cases. All this goes to show that the `oracle'
property, which is based on the pointwise asymptotic distribution only,
paints a highly misleading picture of the behavior of the hard-thresholding
estimator and should not be taken at face value.\footnote{%
This is of course not new and has been observed more than 50 years ago in
the context of Hodges' estimator. For more discussion of the problematic
nature of the `oracle' property see Leeb and P\"{o}tscher (2008a).} A result
for a certain post-model-selection estimator that is related to Theorem \ref%
{H_consist} above can be found in Appendix A of Leeb and P\"{o}tscher (2005).

\begin{theorem}
\label{S_consist} Consider the soft-thresholding estimator with $\eta
_{n}\rightarrow 0$ and $n^{1/2}\eta _{n}\rightarrow \infty $. Assume that $%
n^{1/2}\theta _{n}\rightarrow \nu \in \mathbb{R}\cup \{-\infty ,\infty \}$.
Then $F_{S,n,\theta _{n}}$ approaches pointmass at $-\nu $. In case $|\nu
|<\infty $, this means that $F_{S,n,\theta _{n}}$ converges weakly to
pointmass at $-\nu $; in case $|\nu |=\infty $, it means that the total mass
of $F_{S,n,\theta _{n}}$ escapes to $-\nu $, in the sense that $%
F_{S,n,\theta _{n}}(x)\rightarrow 0$ for every $x\in \mathbb{R}$ if $-\nu
=\infty $, and $F_{S,n,\theta _{n}}(x)\rightarrow 1$ for every $x\in \mathbb{%
R}$ if $-\nu =-\infty $.
\end{theorem}

\begin{proof}
From (\ref{distri_S_1}) we have that $F_{S,n,\theta _{n}}(x)=\Phi
(x+n^{1/2}\eta _{n})$ for $x>-n^{1/2}\theta _{n}$ and $F_{S,n,\theta
_{n}}(x)=\Phi (x-n^{1/2}\eta _{n})$ for $x<-n^{1/2}\theta _{n}$. Because $%
n^{1/2}\eta _{n}\rightarrow \infty $, this entails that $F_{S,n,\theta
_{n}}(x)$ converges to one for each $x>-\nu $ and to zero for each $x<-\nu $.
\end{proof}

The fixed-parameter asymptotic distribution of the soft-thresholding
estimator is obtained by setting $\theta _{n}\equiv \theta $ in the above
theorem: It is $N(0,0)$ (i.e., pointmass at $0$) if $\theta =0$; if $\theta
\neq 0$ the total mass of the finite-sample distribution escapes to $%
\limfunc{sign}(-\theta )\infty $. Hence, the soft-thresholding estimator
when tuned to act as a consistent model selector is not even pointwise $%
n^{1/2}$-consistent (Zou (2006)) and certainly does not satisfy the `oracle'
property. [This contradicts an incorrect claim in Zhao and Yu (2006, Section
2.1) to the effect that tuning LASSO to act as a consistent model selector
results in an asymptotically normal estimator.] The fact that this estimator
is not pointwise $n^{1/2}$-consistent also suggest studying the asymptotic
distribution under a scaling that increases slower than $n^{1/2}$, an issue
that we take up further below; cf. also the appendix.

\begin{theorem}
\label{SCAD_consist} Consider the SCAD estimator with $\eta _{n}\rightarrow
0 $ and $n^{1/2}\eta _{n}\rightarrow \infty $. Assume that $\theta _{n}/\eta
_{n}\rightarrow \zeta $ for some $\zeta \in {\mathbb{R}}\cup \{-\infty
,\infty \}$ and that $n^{1/2}\theta _{n}\rightarrow \nu $ for some $\nu \in {%
\mathbb{R}}\cup \{-\infty ,\infty \}$. [Note that in case $\zeta \neq 0$ the
convergence of $n^{1/2}\theta _{n}$ already follows from that of $\theta
_{n}/\eta _{n}$, and $\nu $ is then given by $\nu =\limfunc{sign}(\zeta
)\infty $.]

\begin{enumerate}
\item If $|\zeta |<a$, or if $|\zeta |=a$ and $n^{1/2}(a\eta _{n}-\limfunc{%
sign}(\zeta )\theta _{n})\rightarrow \infty $, then $F_{SCAD,n,\theta _{n}}$
approaches pointmass at $-\nu $. In case $|\nu |<\infty $, this means that $%
F_{SCAD,n,\theta _{n}}$ converges weakly to pointmass at $-\nu $; in case $%
|\nu |=\infty $, it means that the total mass of $F_{SCAD,n,\theta _{n}}$
escapes to $-\nu $, in the sense that $F_{SCAD,n,\theta _{n}}(x)\rightarrow
0 $ for every $x\in \mathbb{R}$ if $-\nu =\infty $, and $F_{SCAD,n,\theta
_{n}}(x)\rightarrow 1$ for every $x\in \mathbb{R}$ if $-\nu =-\infty $.

\item If $|\zeta |=a$ and $n^{1/2}(a\eta _{n}-\limfunc{sign}(\zeta )\theta
_{n})\rightarrow r$ for some $r\in \mathbb{R\cup \{-\infty \}}$, then $%
F_{SCAD,n,\theta _{n}}(x)$ converges to%
\begin{equation*}
\begin{split}
\int_{-\infty }^{x}\frac{a-2}{a-1}& \phi \Big(\left\{ (a-2)u+\limfunc{sign}%
(\zeta )r\right\} /(a-1)\Big)\;\boldsymbol{1}(\limfunc{sign}(\zeta )u\leq r)
\\
& +\;\;\phi (u)\;\boldsymbol{1}(\limfunc{sign}(\zeta )u>r)\;du
\end{split}%
\end{equation*}%
for every $x\in \mathbb{R}$, with the convention that the integral over the
first term in the above expression is zero if $r=-\infty $. [In fact,
convergence is in total variation distance.]

\item If $a<|\zeta |\leq \infty $, then $F_{SCAD,n,\theta _{n}}$ converges
weakly to the standard normal distribution $N(0,1)$. [In fact, convergence
is in total variation distance.]
\end{enumerate}
\end{theorem}

\begin{proof}
For each $\theta $, the cdf $F_{SCAD,n,\theta }$ consists of contributions
from the atomic part and from the absolutely continuous part. The
contribution of the absolutely continuous part can be further broken down
into the contributions from the integrands $f_{1}$, $f_{2}$, $f_{3}$, $%
f_{-1} $, $f_{-2}$, and $f_{-3}$ in view of (\ref{distri_SCAD}). We hence
may write 
\begin{equation}
\begin{split}
F_{SCAD,n,\theta }(x)\quad =\quad F_{0,n,\theta }(x)\;\;& +\;\;F_{1,n,\theta
}(x)+F_{2,n,\theta }(x)+F_{3,n,\theta }(x) \\
& +\;\;F_{-1,n,\theta }(x)+F_{-2,n,\theta }(x)+F_{-3,n,\theta }(x),
\end{split}
\notag
\end{equation}%
where $F_{0,n,\theta }$ denotes the contribution of the atomic part, and
where the remaining terms on the right-hand side denote the contributions
corresponding to $f_{1}$, $f_{2}$, $f_{3}$, $f_{-1}$, $f_{-2}$, and $f_{-3}$%
, respectively; e.g., $F_{1,n,\theta }(x)=\int_{-\infty }^{x}f_{1}(u)du$.
Now $F_{1,n,\theta _{n}}(x)$ can be written as 
\begin{equation*}
F_{1,n,\theta _{n}}(x)\quad =\quad \int_{-\infty }^{x+n^{1/2}\eta _{n}}\phi
(z)\;\boldsymbol{1}\left( n^{1/2}(\eta _{n}-\theta _{n})<z\leq n^{1/2}(2\eta
_{n}-\theta _{n})\right) \;dz.
\end{equation*}%
[Use the formula for $f_{1}(u)$ given after (\ref{distri_SCAD}) with $\theta
_{n}$ in place of $\theta $, and perform a simple change of variables.] By a
similar argument, we also have 
\begin{equation}
\begin{split}
F_{2,n,\theta _{n}}(x)\;\;=\;\;& \int_{-\infty }^{\left(
(a-2)x+n^{1/2}(a\eta _{n}-\theta _{n})\right) /(a-1)}\phi (z) \\
& \qquad \times \quad \boldsymbol{1}\left( n^{1/2}(2\eta _{n}-\theta
_{n})<z\leq n^{1/2}(a\eta _{n}-\theta _{n})\right) \,dz,
\end{split}
\notag
\end{equation}%
and $F_{3,n,\theta _{n}}(x)=\int_{-\infty }^{x}\phi (z)\boldsymbol{1}\left(
z>n^{1/2}(a\eta _{n}-\theta _{n})\right) dz$.

Assume first that $0\leq \zeta <a$. In the subcase $0\leq \zeta <1$,
Proposition \ref{selection_prob} shows that the total mass of the atomic
part of $F_{SCAD,n,\theta _{n}}$ converges to one, and the statement in part
1 then follows, since $n^{1/2}\theta _{n}\rightarrow \nu $. For the
remaining subcases to be considered observe that we have $n^{1/2}\theta
_{n}\rightarrow \nu =\infty $ whenever $\zeta >0$. For the subcase $\zeta =1$%
, assume for now also that $n^{1/2}(\eta _{n}-\theta _{n})\rightarrow r\in {%
\mathbb{R}}\cup \{-\infty ,\infty \}$. Then the atomic part of $%
F_{SCAD,n,\theta _{n}}$ escapes to $-\nu =-\infty $, and the total mass of
the atomic part converges to $\Phi (r)$ by Proposition~\ref{selection_prob}.
In other words, $F_{0,n,\theta _{n}}(x)\rightarrow \Phi (r)$ for each $x\in 
\mathbb{R}$, where $F_{0,n,\theta _{n}}$ denotes the contribution from the
atomic part of $F_{SCAD,n,\theta _{n}}$. Moreover, from the preceding
formula for $F_{1,n,\theta _{n}}(x)$, it is evident that $F_{1,n,\theta
_{n}}(x)\rightarrow \int_{r}^{\infty }\phi (z)dz=1-\Phi (r)$ holds for each $%
x\in \mathbb{R}$ (because the upper limit in the integral diverges to $%
\infty $, because the lower limit in the indicator is $n^{1/2}(\eta
_{n}-\theta _{n})\rightarrow r$, and because the upper limit is $%
n^{1/2}(2\eta _{n}-\theta _{n})\rightarrow \infty $). Hence, $%
F_{SCAD,n,\theta _{n}}(x)\geq F_{0,n,\theta _{n}}(x)+F_{1,n,\theta
_{n}}(x)\rightarrow 1$ for each $x\in \mathbb{R}$, as required. Because that
limit does not depend on $r$, and because any subsequence contains a further
subsequence along which $n^{1/2}(\eta _{n}-\theta _{n})$ converges to some
limit $r\in {\mathbb{R}}\cup \{-\infty ,\infty \}$ (due to compactness of
this space), the result follows for the subcase $\zeta =1$. In the subcase $%
1<\zeta <2$, it is easy to see that $F_{1,n,\theta _{n}}(x)$ converges to
one for each $x\in \mathbb{R}$, whence $F_{SCAD,n,\theta _{n}}(x)\geq
F_{1,n,\theta _{n}}(x)\rightarrow 1$ for each $x\in \mathbb{R}$. In the
subcase $\zeta =2$, assume for now also that $n^{1/2}(2\eta _{n}-\theta
_{n})\rightarrow r\in {\mathbb{R}}\cup \{-\infty ,\infty \}$. We then see
that $F_{1,n,\theta _{n}}(x)\rightarrow \Phi (r)$ and $F_{2,n,\theta
_{n}}(x)\rightarrow 1-\Phi (r)$, whence $F_{SCAD,n,\theta
_{n}}(x)\rightarrow 1$ for each $x\in \mathbb{R}$. Because this limit does
not depend on $r$ and ${\mathbb{R}}\cup \{-\infty ,\infty \}$ is compact, a
subsequence argument as above shows that the statement follows also in this
subcase. Finally, in the subcases $2<\zeta <a$ and $\zeta =a$ but $%
n^{1/2}(a\eta _{n}-\theta _{n})\rightarrow \infty $, it suffices to note
that $F_{2,n,\theta _{n}}(x)\rightarrow 1$ for all $x\in \mathbb{R}$.

Assume next that $\zeta =a$ and that $n^{1/2}(a\eta _{n}-\theta
_{n})\rightarrow r\in {\mathbb{R}}\cup \{-\infty \}$. Note that $%
n^{1/2}(2\eta _{n}-\theta _{n})=n^{1/2}\eta _{n}(2-\theta _{n}/\eta
_{n})\rightarrow -\infty $ holds because $\zeta =a>2$. Using the formula for 
$f_{2}(u)$ and $f_{3}(u)$ given after (\ref{distri_SCAD}) with $u$ replacing 
$x$ and $\theta _{n}$ replacing $\theta $, it is then easy to see that $%
f_{2}(u)+f_{3}(u)$ converges to the integrand in the display given in part
2, for almost all $u$. Moreover, the total mass of $F_{2,n,\theta
_{n}}+F_{3,n,\theta _{n}}$ is also easily computed and seen to converge to
one. Furthermore, it is easily checked that the total mass of the limiting
cdf displayed in part 2 is one. Scheff\'{e}'s Lemma then shows that $%
F_{2,n,\theta _{n}}+F_{3,n,\theta _{n}}$, and hence $F_{SCAD,n,\theta _{n}}$%
, converge in total variation to the limit cdf given in part 2.

Next, assume that $\zeta >a$. Then the integrand in the formula for $%
F_{3,n,\theta _{n}}(x)$ converges to the density $\phi (z)$ for each $z$.
The Dominated Convergence Theorem then establishes the convergence of $%
F_{3,n,\theta _{n}}$, and hence of $F_{SCAD,n,\theta _{n}}$, to $\Phi $ in
total variation distance.

For $\zeta <0$, the proof is, mutatis mutandis, the same with $f_{-1}$, $%
f_{-2}$, and $f_{-3}$ now taking the roles of $f_{1}$, $f_{2}$, and $f_{3}$,
respectively, and with the case $-a<\zeta \leq -1$ now being handled by
showing that $1-F_{SCAD,n,\theta _{n}}(x)\rightarrow 1$ for each $x\in 
\mathbb{R}$. Alternatively, it can be reduced to what has already been
established by observing that $F_{SCAD,n,\theta _{n}}(x)=1-F_{SCAD,n,-\theta
_{n}}(-x-)$, where $F_{SCAD,n,-\theta _{n}}(\cdot -)$ denotes the limit from
the left of $F_{SCAD,n,-\theta _{n}}$ at the indicated argument.
\end{proof}

The fixed-parameter asymptotic distribution of the SCAD estimator, including
the `oracle' property discussed at the beginning of this section, can
clearly be recovered from Theorem \ref{SCAD_consist} by setting $\theta
_{n}\equiv \theta $. Like in the case of the hard-thresholding estimator,
Theorem \ref{SCAD_consist}\ shows that the asymptotic behavior of the
SCAD-estimator is much more complicated than what the `oracle' property
predicts. In particular, Theorem \ref{SCAD_consist} shows that the
SCAD-estimator is not uniformly $n^{1/2}$-consistent. [For a discussion of
the behavior of this estimator under a different scaling see the next
paragraph as well as the appendix.] Furthermore, since the finite-sample
distribution of the SCAD-estimator is highly non-normal but the \textit{%
pointwise} asymptotic distribution is normal, the latter cannot adequately
capture many of the essential features of the former. In contrast, the
asymptotic distributions given in Theorem \ref{SCAD_consist} are non-normal
in some cases. All this again shows that the `oracle' property is more of an
artifact of the asymptotic framework than of much statistical significance.

The observation, that the estimators $\hat{\theta}_{H}$ and $\hat{\theta}%
_{SCAD}$ are not uniformly $n^{1/2}$-consistent if tuned to perform
consistent model selection, prompts the question of the behavior of $c_{n}(%
\hat{\theta}_{H}-\theta )$ and $c_{n}(\hat{\theta}_{SCAD}-\theta )$ under a
sequence of norming constants $c_{n}$ that are $o(n^{1/2})$. Since both
estimators are \emph{pointwise} $n^{1/2}$-consistent, it follows that the 
\textit{pointwise} limiting distributions of $c_{n}(\hat{\theta}_{H}-\theta )
$ and $c_{n}(\hat{\theta}_{SCAD}-\theta )$ will then degenerate to pointmass
at zero. Furthermore, it is not difficult to see that under general `moving
parameter' asymptotics the finite-sample distributions of $c_{n}(\hat{\theta}%
_{H}-\theta _{n})$ and $c_{n}(\hat{\theta}_{SCAD}-\theta _{n})$ are then
nevertheless stochastically unbounded for certain sequences of parameters $%
\theta _{n}$ unless $c_{n}=O(\eta _{n}^{-1})$. If $c_{n}=O(\eta _{n}^{-1})$,
Theorem \ref{unif_cons} has shown that $c_{n}(\hat{\theta}_{H}-\theta _{n})$
and $c_{n}(\hat{\theta}_{SCAD}-\theta _{n})$ are indeed stochastically
bounded. Hence, the uniform convergence rate of $\hat{\theta}_{H}$ and $\hat{%
\theta}_{SCAD}$ is seen to be given precisely by $\eta _{n}$. The precise
limit distributions of these estimators under a scaling by $c_{n}$ can be
obtained in a manner similar to the above theorems and are given in Theorems %
\ref{hard_rescaled} and \ref{SCAD_rescaled} in the appendix for the (only
interesting) case $c_{n}=\eta _{n}^{-1}$. It turns out that the limit
distributions under `moving parameter' asymptotics are always given by a
linear combination of at most two pointmasses, each located in the interval $%
[-1,1]$. With regard to the soft-thresholding estimator we have already
observed that it is not even pointwise $n^{1/2}$-consistent. Even the
distributions of $c_{n}(\hat{\theta}_{S}-\theta )$ with $c_{n}=o(n^{1/2})$
are stochastically unbounded if $\theta \neq 0$ unless $c_{n}=O(\eta
_{n}^{-1})$. This is most easily seen by using the relation to the
hard-thresholding estimator given in (\ref{hard-soft}). If $c_{n}=O(\eta
_{n}^{-1})$, relation (\ref{hard-soft}) also shows that $c_{n}(\hat{\theta}%
_{S}-\theta )$ is stochastically bounded, but has a degenerate (pointwise)
limiting distribution. This has been noted by Zou (2006). In view of Theorem %
\ref{unif_cons}, under this condition on $c_{n}$ the distributions of $c_{n}(%
\hat{\theta}_{S}-\theta _{n})$ are in fact stochastically bounded for 
\textit{any} sequence $\theta _{n}$. The precise forms of the possible limit
distributions under such a `moving parameter' asymptotic are given in
Theorem \ref{soft_rescaled} in the appendix.

Theorems \ref{H_consist} and \ref{SCAD_consist} demonstrate that the
`oracle' property of the hard-thresholding estimator and of the
SCAD-estimator paints a misleading picture of the actual finite-sample
behavior of these estimators due to nonuniformity problems. In order to
rescue the `oracle' property, sometimes the argument is put forward that
parameter sequences $\theta _{n}$ that are responsible for the nonuniformity
problem should be eliminated from the parameter space a priori, since such $%
\theta _{n}$ are supposedly close to zero and hence are difficult to
distinguish statistically from zero. While we think that such a reasoning is
not sensible (because asymptotic properties of statistical procedures that
are quite unstable under local perturbations of the parameter are highly
suspect), we next show that the suggested reasoning actually is flawed:
Consider first the consistently tuned SCAD-estimator. Suppose one considers
a priori the restricted parameter space $\Theta _{n}$\ of the form $\Theta
_{n}=\left\{ \theta :\theta =0\text{ or }\left\vert \theta \right\vert \geq
b_{n}\right\} $ for some sequence $b_{n}>0$. In order to achieve that for
every $\theta _{n}\in \Theta _{n}$ with $\theta _{n}\neq 0$, the
distribution of $n^{1/2}(\hat{\theta}_{SCAD}-\theta _{n})$ converges weakly
to the standard normal $N(0,1)$ (as desired when attempting to rescue the
`oracle' property), it follows from Theorem \ref{SCAD_consist} that $b_{n}$
would have to satisfy $n^{1/2}\eta _{n}(a-b_{n}/\eta _{n})\rightarrow
-\infty $ (e.g., $b_{n}\equiv b\eta _{n}$ with $b>a$). But then it is easy
to see that the `forbidden' set $\mathbb{R}\diagdown \Theta _{n}$ contains
elements $\theta _{n}$ that are large in the sense that (i) they are of
order larger than $n^{-1/2}$ and (ii) they are classified as nonzero with
probability converging to unity by the very same SCAD-procedure, i.e., $%
P_{n,\theta _{n}}(\hat{\theta}_{SCAD}\neq 0)\rightarrow 1$ holds (to see
this use Proposition \ref{selection_prob}). On top of this, the parameter
space $\Theta _{n}$\ is highly artificial, depends on sample size, and also
on the tuning parameter $\eta _{n}$ and thus on the estimation procedure
used. An analogous statement holds for the hard-thresholding estimator (with
the exception that the `forbidden' set in this case contains $\theta _{n}$
that are large in the sense that they satisfy (i) above and (ii) are
classified as non-zero with probability tending to unity by any \textit{%
conservatively} tuned hard-thresholding procedure). Taken together, this
shows that adopting a parameter space like $\Theta _{n}$ rules out values of 
$\theta $ that are substantially large, and not only values of $\theta $
that are statistically difficult to distinguish from zero. Hence, there
seems to be little support for adopting such $\Theta _{n}$ as the parameter
space.

\begin{remark}
\normalfont\label{rxxx} The theorems in this subsection actually completely
describe the limiting behavior of the finite-sample distributions of $\hat{%
\theta}_{H}$, $\hat{\theta}_{S}$, and $\hat{\theta}_{SCAD}$ without \emph{%
any }condition on the sequence of parameters $\theta _{n}$. To see this,
just apply the theorems to subsequences and note that by compactness of ${%
\mathbb{R}}\cup \{-\infty ,\infty \}$ we can select from each subsequence a
further subsequence such that the relevant quantities like $n^{1/2}\theta
_{n}$, $\theta _{n}/\eta _{n}$, $n^{1/2}(\eta _{n}-\theta _{n})$, $%
n^{1/2}(\eta _{n}+\theta _{n})$, etc. converge in ${\mathbb{R}}\cup
\{-\infty ,\infty \}$ along this further subsequence.
\end{remark}

\begin{remark}
\normalfont\label{ry} (i) As a point of interest we note that the full
complexity of the possible limiting distributions in Theorems \ref{H_consist}%
, \ref{S_consist}, and \ref{SCAD_consist} already arises if we restrict the
sequences $\theta _{n}$ to a bounded neighborhood of zero. Hence, the
phenomena described by these theorems are of a local nature, and are not
tied in any way to the unboundedness of the parameter space.

(ii) It is also interesting to observe that what governs the different
cases, in Theorems \ref{H_consist} and \ref{SCAD_consist}, is essentially
the behavior of $\theta _{n}/\eta _{n}$, which is of smaller order than $%
n^{1/2}\theta _{n}$ because $n^{1/2}\eta _{n}\rightarrow \infty $ in the
consistent case. Hence, an analysis relying on the usual local asymptotics
based on perturbations of $\theta $ of the order of $n^{-1/2}$ does not
properly reveal all possible limits of the finite-sample distributions in
the case where the estimators perform consistent model selection.
\end{remark}

\begin{remark}
\normalfont\label{ryy} Similar as in Section \ref{sec_conserv}, the
mathematical reason for the failure of the pointwise asymptotic distribution
to capture the behavior of the finite-sample distribution well is that the
convergence of the latter to the former is not uniform in the underlying
parameter $\theta $. See Leeb and P\"{o}tscher (2003, 2005) for more
discussion in the context of post-model-selection estimators.
\end{remark}

\section{Impossibility results for estimating the distribution of $\hat{%
\protect\theta}_{H}$, $\hat{\protect\theta}_{S}$, and $\hat{\protect\theta}%
_{SCAD}$\label{imposs}}

As shown in Section \ref{finite}, the cdfs $F_{H,n,\theta }$, $F_{S,n,\theta
}$, and $F_{SCAD,n,\theta }$ of the (centered and scaled) estimators $\hat{%
\theta}_{H}$, $\hat{\theta}_{S}$, and $\hat{\theta}_{SCAD}$ depend on the
unknown parameter $\theta $ in a complicated manner. It is hence of interest
to consider estimation of these cdfs. We show that this is an intrinsically
difficult estimation problem in the sense that these cdfs can not be
estimated in a uniformly consistent fashion. Parts of the results that
follow have been presented in earlier work (in slightly different settings):
For a general class of post-model-selection estimators including the
hard-thresholding estimator, this phenomenon was discussed in Leeb and P\"{o}%
tscher (2006b,2008b) for the case where the estimator is tuned to be
conservative, whereas Leeb and P\"{o}tscher (2006a) consider the case where
the hard-thresholding estimator is tuned to be consistent; the latter paper
also gives similar results for a soft-thresholding estimator tuned to be
conservative. In the following, we give a simple unified treatment of
hard-thresholding, soft-thresholding, and also of the SCAD estimator. For
the SCAD estimator and for the consistently tuned soft-thresholding
estimator, such non-uniformity phenomena in estimating the estimator's cdf
have not been established before. We provide large-sample results that cover
both consistent and conservative choices of the tuning parameter, as well as
finite-sample results that hold for any choice of tuning parameter.

It is straight-forward to construct consistent estimators for the
distributions of the (centered and scaled) estimators $\hat{\theta}_{H}$, $%
\hat{\theta}_{S}$ and $\hat{\theta}_{SCAD}$. One popular choice is to use
subsampling or the $m$ out of $n$ bootstrap with $m/n\rightarrow 0$. Another
possibility is to use the pointwise large-sample limit distributions derived
in Section~\ref{asydistribs} together with a properly chosen pre-test of the
hypothesis $\theta =0$ versus $\theta \neq 0$: Because the pointwise
large-sample limit distribution takes only two different functional forms
depending on whether $\theta =0$ or $\theta \neq 0$, one can perform a
pre-test that rejects the hypothesis $\theta =0$ in case $|\bar{y}|>n^{-1/4}$%
, say, and estimate the finite-sample distribution by that large-sample
limit formula that corresponds to the outcome of the pre-test;\footnote{%
In the consevative case, the asymptotic distribution can also depend on $e$
which is then to be replaced by $n^{1/2}\eta _{n}$.} the test's critical
value $n^{-1/4}$ ensures that the correct large-sample limit formula is
selected with probability approaching one as sample size increases.

When estimating the distribution of thresholding (and related) estimators,
there is evidence in the literature that certain specific consistent
estimation procedures, like those sketched above, may not perform well in a
worst-case scenario. For some examples, see Kulperger and Ahmed (1992); the
disclaimer issued after Corollary~2.1 in Beran (1997); the discussion at the
end of Section~4 in Knight and Fu (2000); or Samworth (2003). The next
result shows that this problem is not caused by the specifics of the
consistent estimators under consideration but is an intrinsic feature of the
estimation problem itself.

\begin{theorem}
\label{t1} Let $\hat{\theta}$ denote any one of the estimators $\hat{\theta}%
_{H}$, $\hat{\theta}_{S}$, or $\hat{\theta}_{SCAD}$, and write $F_{n,\theta
} $ for the cdf of $n^{1/2}(\hat{\theta}-\theta )$ under $P_{n,\theta }$.
Consider a sequence of tuning parameters such that $\eta _{n}\rightarrow 0$
and $n^{1/2}\eta _{n}\rightarrow e$ as $n\rightarrow \infty $ with $0\leq
e\leq \infty $. Let $t\in \mathbb{R}$ be arbitrary. Then \emph{every
consistent }estimator\emph{\ }$\hat{F}_{n}(t)$ of $F_{n,\theta }(t)$
satisfies 
\begin{equation}
\lim_{n\rightarrow \infty }\sup_{|\theta |<c/n^{1/2}}P_{n,\theta }\left(
\left\vert \hat{F}_{n}(t)-F_{n,\theta }(t)\right\vert \;>\;\varepsilon
\right) \quad =\quad 1  \notag
\end{equation}%
for each $\varepsilon <(\Phi (t+e)-\Phi (t-e))/2$ and each $c>\left\vert
t\right\vert $.\emph{\ }In particular, no uniformly consistent estimator for 
$F_{n,\theta }(t)$ exists.
\end{theorem}

\begin{proof}
For two sequences $\theta _{n}^{(1)}$ and $\theta _{n}^{(2)}$ satisfying $%
|\theta _{n}^{(i)}|<c/n^{1/2}$, $i=1,2$, the probability measures $%
P_{n,\theta _{n}^{(1)}}$ and $P_{n,\theta _{n}^{(2)}}$ are mutually
contiguous as is elementary to verify (cf., e.g., Lemma~A.1 of Leeb and P%
\"{o}tscher (2006a)). The corresponding estimands $F_{n,\theta
_{n}^{(1)}}(t) $ and $F_{n,\theta _{n}^{(2)}}(t)$, however, do not
necessarily get close to each other: For each $\delta $ write $\theta
_{n}(\delta )$ as shorthand for $\theta _{n}(\delta )=-(t+\delta )/n^{1/2}$.
The cdfs $F_{n,\theta _{n}(\delta )}(\cdot )$ and $F_{n,\theta _{n}(-\delta
)}(\cdot )$ have a jump at $t+\delta $ and at $t-\delta $, respectively, so
that for $\delta >0$%
\begin{equation}
F_{n,\theta _{n}(-\delta )}(t)-F_{n,\theta _{n}(\delta )}(t)\quad =\quad
\Phi (t-\delta +n^{1/2}\eta _{n})-\Phi (t-\delta -n^{1/2}\eta
_{n})\;+\;r(\delta );  \label{pt1.1}
\end{equation}%
cf. (\ref{distri_H}), (\ref{distri_S}), and (\ref{distri_SCAD}) for $\hat{%
\theta}_{H}$, $\hat{\theta}_{S}$, and $\hat{\theta}_{SCAD}$, respectively.
Moreover, $r(\delta )$ goes to zero with $\delta \rightarrow 0$, because the
absolutely continuous part of $F_{n,\theta }(t)$ is a continuous function of 
$\theta $ (again in view of the finite-sample formulae and dominated
convergence). Taking the supremum of $\left\vert F_{n,\theta _{n}(-\delta
)}(t)-F_{n,\theta _{n}(\delta )}(t)\right\vert $ over all $\delta $ with $%
0\leq \delta <c-|t|$, we obtain that this supremum is bounded from below by $%
\Phi (t+n^{1/2}\eta _{n})-\Phi (t-n^{1/2}\eta _{n})$. [To see this note that
this supremum is not less than $\lim_{i\rightarrow \infty }\left\vert
F_{n,\theta _{n}(-1/i)}(t)-F_{n,\theta _{n}(1/i)}(t)\right\vert $ and use (%
\ref{pt1.1}).] Because that lower bound converges to $\Phi (t+e)-\Phi (t-e)$
as $n\rightarrow \infty $, the theorem now follows from Lemma~3.1 of Leeb
and P\"{o}tscher (2006a). [Use this result with the identifications $\beta
=\theta $, $\varphi _{n}(\beta )=F_{n,\theta }(t)$, $B_{n}=\{\theta
:\;|\theta |<c/n^{1/2}\}$, $\alpha =0$, and with $d(a,b)=|a-b|$. Moreover,
note that $B_{n}$ contains $\theta _{n}(\delta )$ and $\theta _{n}(-\delta )$
for $0\leq \delta <c-|t|$.]
\end{proof}

We stress that the above result also applies to any kind of bootstrap- or
subsampling-based estimator of the cdf $F_{n,\theta }$ whatsoever, since the
results in Leeb and P\"{o}tscher (2006a) on which the proof of Theorem~\ref%
{t1} rests apply to arbitrary randomized estimators (cf. Lemma 3.6 in Leeb
and P\"{o}tscher (2006a)); the same applies to Theorem~\ref{t2} that follows
as well as to Theorem \ref{t3}\ in the appendix.

Loosely speaking, Theorem~\ref{t1} states that any consistent estimator for
the cdf of interest suffers from an unavoidable worst-case error of at least 
$\varepsilon $ with $\varepsilon <(\Phi (t+e)-\Phi (t-e))/2$. The error
range, i.e., $(\Phi (t+e)-\Phi (t-e))/2$, is governed by the limit $%
e=\lim_{n}n^{1/2}\eta _{n}$. In case the estimator $\hat{\theta}$ is tuned
to be consistent, i.e., in case $e=\infty $, the error range equals $1/2$,
and the phenomenon is most pronounced. If the estimator $\hat{\theta}$ is
tuned to be conservative so that $e<\infty $, the error range is less than $%
1/2$ but can still be substantial. Only in case $e=0$ the error range equals
zero, and the condition $\varepsilon <(\Phi (t+e)-\Phi (t-e))/2$ in Theorem~%
\ref{t1} leads to a trivial conclusion. This is, however, not surprising as
then the resulting estimator is uniformly asymptotically equivalent to the
unrestricted maximum likelihood estimator $\bar{y}$; cf. Remark \ref{r1}.

A similar non-uniformity phenomenon as described in Theorem~\ref{t1} for
consistent estimators $\hat{F}_{n}(t)$ also occurs for not necessarily
consistent estimators. For such arbitrary estimators, we find in the
following that the phenomenon can be somewhat less pronounced, in the sense
that the lower bound is now $1/2$ instead of $1$; cf. (\ref{t2.1}) below.
The following theorem gives a large-sample limit result that parallels
Theorem~\ref{t1}, as well as a finite-sample result, both for arbitrary (and
not necessarily consistent) estimators of the cdf.

\begin{theorem}
\label{t2} Let $\hat{\theta}$ denote any one of the estimators $\hat{\theta}%
_{H}$, $\hat{\theta}_{S}$, or $\hat{\theta}_{SCAD}$, and write $F_{n,\theta
} $ for the cdf of $n^{1/2}(\hat{\theta}-\theta )$ under $P_{n,\theta }$.
Let $0<\eta _{n}<\infty $ and let $t\in \mathbb{R}$ be arbitrary. Then \emph{%
every} estimator $\hat{F}_{n}(t)$ of $F_{n,\theta }(t)$ satisfies 
\begin{equation}
\sup_{|\theta |<c/n^{1/2}}P_{n,\theta }\left( \left\vert \hat{F}%
_{n}(t)-F_{n,\theta }(t)\right\vert \;>\;\varepsilon \right) \quad \geq
\quad \frac{1}{2}  \label{t2.2}
\end{equation}%
for each $\varepsilon <(\Phi (t+n^{1/2}\eta _{n})-\Phi (t-n^{1/2}\eta
_{n}))/2$, for each $c>\left\vert t\right\vert $, and for each fixed sample
size $n$. If $\eta _{n}$ satisfies $\eta _{n}\rightarrow 0$ and $n^{1/2}\eta
_{n}\rightarrow e$ as $n\rightarrow \infty $ with $0\leq e\leq \infty $, we
thus have 
\begin{equation}
\liminf_{n\rightarrow \infty }\inf_{\hat{F}_{n}(t)}\sup_{|\theta
|<c/n^{1/2}}P_{n,\theta }\left( \left\vert \hat{F}_{n}(t)-F_{n,\theta
}(t)\right\vert \;>\;\varepsilon \right) \quad \geq \quad \frac{1}{2}
\label{t2.1}
\end{equation}%
for each $\varepsilon <(\Phi (t+e)-\Phi (t-e))/2$ and for each $c>\left\vert
t\right\vert $, where the infimum in (\ref{t2.1}) extends over \emph{all}
estimators $\hat{F}_{n}(t)$.
\end{theorem}

\begin{proof}
Only the finite-sample statement needs to be proven. Let $\theta _{n}(\delta
)$ be as in the proof of Theorem~\ref{t1}. The total variation distance of $%
P_{n,\theta _{n}(\delta )}$ and $P_{n,\theta _{n}(-\delta )}$, i.e., $%
||P_{n,\theta _{n}(\delta )}-P_{n,\theta _{n}(-\delta )}||_{TV}$, goes to
zero as $\delta \rightarrow 0$ (which is easy to see, either by direct
computation or using, say, Lemma~A.1 of Leeb and P\"{o}tscher (2006a)). In
view of (\ref{pt1.1}), however, the estimands $F_{n,\theta _{n}(\delta )}(t)$
and $F_{n,\theta _{n}(-\delta )}(t)$ do not get close to each other as $%
\delta \rightarrow 0$ ($\delta >0$), as we have already seen in the proof of
Theorem~\ref{t1}. For each $\varepsilon $ that is smaller than $\left\vert
F_{n,\theta _{n}(-\delta )}(t)-F_{n,\theta _{n}(\delta )}(t)\right\vert /2$,
the left-hand side of (\ref{t2.2}) is bounded from below by 
\begin{equation}
\frac{1}{2}\Big(1-||P_{n,\theta _{n}(\delta )}-P_{n,\theta _{n}(-\delta
)}||_{TV}\Big).  \notag
\end{equation}%
This follows from Lemma~3.2 of Leeb and P\"{o}tscher (2006a) together with
Remark~B.2 of that paper. [Use the result described in Remark~B.2 with $%
A=\{n\}$, $\beta =\theta $, $B_{n}=\{\theta _{n}(\delta ),\theta
_{n}(-\delta )\}$, $\varphi _{n}(\beta )=F_{n,\theta }(t)$, $d(a,b)=|a-b|$,
and with $\delta ^{\ast }$ equal to $\left\vert F_{n,\theta _{n}(-\delta
)}(t)-F_{n,\theta _{n}(\delta )}(t)\right\vert $. Moreover, note that $B_{n}$
is contained in $\{\theta :\left\vert \theta \right\vert <c/n^{1/2}\}$
provided $0<\delta <c-|t|$.] For $\delta \rightarrow 0$, now observe that
the expression in the preceding display converges to $1/2$, i.e., the lower
bound in (\ref{t2.2}), and that $\left\vert F_{n,\theta _{n}(-\delta
)}(t)-F_{n,\theta _{n}(\delta )}(t)\right\vert $ converges to $\Phi
(t+n^{1/2}\eta _{n})-\Phi (t-n^{1/2}\eta _{n})$.
\end{proof}

Apart from being of interest in its own right, the asymptotic statement in
Theorem~\ref{t2} also provides additional insight into some phenomena
related to inference based on shrinkage-type estimators that have recently
attracted some attention: When estimating the cdf of a hard-thresholding
estimator, Samworth (2003) noted that, while the bootstrap is not
consistent, it nevertheless may perform better, in a uniform sense, than the 
$m$ out of $n$ bootstrap which is consistent (provided $m\rightarrow \infty $%
, $m/n\rightarrow 0$). Theorem~\ref{t1} and the asymptotic statement in
Theorem~\ref{t2} together show that this phenomenon of better performance of
the bootstrap is possible precisely because the bootstrap is not consistent.

The finite-sample statement in Theorem~\ref{t2} clearly reveals how the
estimability of the cdf of the estimator depends on the tuning parameter $%
\eta _{n}$: A larger value of $\eta _{n}$, which results in a `more sparse'
estimator in view of (\ref{model_prob}), directly corresponds to a larger
range $(\Phi (t+n^{1/2}\eta _{n})-\Phi (t-n^{1/2}\eta _{n}))/2$ for the
error $\varepsilon $ within which any estimator $\hat{F}_{n}(t)$ performs
poorly in the sense of (\ref{t2.2}). In large samples, the limit $%
e=\lim_{n}n^{1/2}\eta _{n}$ takes the role of $n^{1/2}\eta _{n}$.

An impossibility result paralleling Theorem \ref{t2} for the cdf of $\eta
_{n}^{-1}(\hat{\theta}-\theta )$, where $\hat{\theta}=\hat{\theta}_{H}$, $%
\hat{\theta}_{S}$, or $\hat{\theta}_{SCAD}$, is given in the appendix.

\section{Conclusion\label{conclusion}}

We have studied the distribution of the LASSO, i.e., of a soft-thresholding
estimator, of the SCAD, and of a hard-thresholding estimator in finite
samples and in the large-sample limit. The finite-sample distributions of
these estimators were found to be highly non-normal, because they are a
mixture of a singular normal distribution and an absolutely continuous
component that can be multimodal, for example. The large-sample behavior of
these distributions depends on the choice of the estimators' tuning
parameter where, in essence, two cases can occur:

In the first case, the estimator can be viewed as performing conservative
model selection. In this case, fixed-parameter asymptotics, where the true
parameters are held fixed while sample size increases, reflect the
large-sample behavior only in part. `Moving parameter' asymptotics, where
the true parameter may depend on sample size, give a more complete picture.
We have seen that the distribution of the LASSO, of the SCAD, and of the
hard-thresholding estimator can be highly non-normal irrespective of sample
size, in particular in the statistically interesting case where the true
parameter is close (in an appropriate sense) to a lower-dimensional
submodel. This also shows that the finite-sample phenomena that we have
observed are not small-sample effects but can occur at any sample size.

In the second case, the estimator can be viewed as performing consistent
model selection, and the hard-thresholding as well as the SCAD estimator
have the `oracle' property in the sense of Fan and Li (2001). [This is not
so for the LASSO.] This `oracle' property, which is based on fixed-parameter
asymptotics, seems to suggest that the estimator in question performs very
well in large samples. However, as before, fixed-parameter asymptotics do
not capture the whole range of large-sample phenomena that can occur. With
`moving parameter' asymptotics, we have shown that the distribution of these
estimators can again be highly non-normal, even in large samples. In
addition, we have found that the observed finite-sample phenomena not only
can persist but actually can be more pronounced for larger sample sizes. For
example, the distribution of the SCAD estimator can diverge in the sense
that all its mass escapes to either $+\infty $ or $-\infty $.

We have also demonstrated that the LASSO, the SCAD, and the
hard-thresholding estimator are always uniformly consistent, irrespective of
the choice of tuning parameter (except for non-sensible choices). In case
the tuning is such that the estimator acts as a conservative model selector,
we have also seen that these estimators are in fact uniformly $n^{1/2}$%
-consistent. However, uniform $n^{1/2}$-consistency no longer holds in the
case where the estimator acts like a consistent model selector (and where
the SCAD and the hard-thresholding estimator have the `oracle' property). In
fact, the estimators then have a uniform convergence rate slower than $%
n^{-1/2}$ in that they are only uniformly $\eta _{n}^{-1}$-consistent. The
asymptotic distributions of the estimators under an $\eta _{n}^{-1}$%
-scaling, rather than an $n^{1/2}$-scaling, are discussed in the appendix.

Finally, we have studied the problem of estimating the cdf of the (centered
and scaled) LASSO, SCAD, and hard-thresholding estimator. We have shown that
this cdf can not be estimated in a uniformly consistent fashion, even though
pointwise consistent estimators can be constructed with relative ease.
Moreover, we have obtained performance bounds for estimators of the cdf that
suggest that inconsistent estimators for this cdf may actually perform
better, in a uniform sense, than consistent estimators.

The phenomena observed here for distributional properties of the estimators
under consideration not surprisingly spill over to the estimators' risk
behavior. The finite-sample distributions derived in this paper in fact
facilitate a detailed risk analysis, but this is not our main focus here.
Therefore, we only point out the most important risk phenomena: We consider
squared error loss scaled by sample size (i.e., $L(\hat{\theta},\theta )=n(%
\hat{\theta}-\theta )^{2}$), and we shall compare the estimators to the
maximum-likelihood estimator based on the overall model, i.e., $\hat{\theta}%
_{U}=\bar{y}$. In finite samples, the LASSO, the SCAD, and the
hard-thresholding estimator compare favorably with $\hat{\theta}_{U}$ in
terms of risk, if the true parameter is in a neighborhood of the lower
dimensional model; outside of that neighborhood, the situation is reversed.
[This is well-known for the hard- and soft-thresholding estimators and for
more general pre-test estimators; cf. Judge and Bock (1978), Bruce and Gao
(1996). Explicit formulae for the risk of a hard-thresholding estimator are
also given in Leeb and P\"{o}tscher (2005).] As sample size goes to
infinity, again two cases need to be distinguished: If these estimators are
tuned to perform conservative model selection, the worst-case risk of the
LASSO, of the SCAD, and of the hard-thresholding estimator remains bounded
as sample size increases. If the tuning is such that these estimators
perform consistent model selection (the case when the SCAD as well as the
hard-thresholding estimator have the `oracle' property), then the worst-case
risk of these estimators increases indefinitely as sample size goes to
infinity. [In fact, this is true for any estimator that has a `sparsity'
property; see Theorem~2.1 in Leeb and P\"{o}tscher (2008a) for details.]
Thus for these estimators the asymptotic worst-case risk behavior is in
marked contrast to their favorable pointwise asymptotic risk behavior
reflected in the `oracle' property. For the SCAD, the LASSO, and for the
hard-thresholding estimator, this worst-case risk behavior is also in line
with the fact that these estimators are uniformly $n^{1/2}$-consistent if
tuned to perform conservative model selection, but that uniform $n^{1/2}$%
-consistency breaks down when they are tuned to perform consistent model
selection.

Finally we want to stress that our results should not be read as a criticism
of penalized maximum likelihood estimators per se, but rather as a warning
that the distributional properties of such estimators are more intricate and
complex than might appear at first glance.

\section*{Acknowledgments}

Input and suggestions from Ulrike Schneider are greatly appreciated.

\appendix{}

\section{Appendix}

For the case where the estimators $\hat{\theta}_{H}$, $\hat{\theta}_{S}$,
and $\hat{\theta}_{SCAD}$ are tuned to perform consistent model selection
(i.e., $\eta _{n}\rightarrow 0$ and $n^{1/2}\eta _{n}\rightarrow \infty $),
we now consider the possible limits of the distributions of $c_{n}(\hat{%
\theta}_{H}-\theta _{n})$, $c_{n}(\hat{\theta}_{S}-\theta _{n})$, and $c_{n}(%
\hat{\theta}_{SCAD}-\theta _{n})$ when $c_{n}=O(\eta _{n}^{-1})$. The only
interesting case is where $c_{n}\sim \eta _{n}^{-1}$, since for $%
c_{n}=o(\eta _{n}^{-1})$ these limits are always pointmass at zero in view
of Theorem \ref{unif_cons}.\footnote{%
There is no loss in generality here in the sense that the general case where 
$c_{n}=O(\eta _{n}^{-1})$ holds can -- by passing to subsequences -- always
be reduced to the cases where $c_{n}\sim \eta _{n}^{-1}$ or $c_{n}=o(\eta
_{n}^{-1})$ holds.} Let $G_{H,n,\theta }$, $G_{S,n,\theta }$, and $%
G_{SCAD,n,\theta }$ stand for the finite-sample distributions of $\eta
_{n}^{-1}(\hat{\theta}_{H}-\theta )$, $\eta _{n}^{-1}(\hat{\theta}%
_{S}-\theta )$, and $\eta _{n}^{-1}(\hat{\theta}_{SCAD}-\theta )$,
respectively, under $P_{n,\theta }$. Clearly, $G_{H,n,\theta
}(x)=F_{H,n,\theta }(n^{1/2}\eta _{n}x)$ and similar relations hold for $%
G_{S,n,\theta }$ and $G_{SCAD,n,\theta }$. We next provide the limits of
these distributions under `moving parameter' asymptotics. Note that comments
like in Remarks \ref{rxxx}, \ref{ry}, and \ref{ryy} also apply to the three
subsequent theorems.

\begin{theorem}
\label{hard_rescaled}Consider the hard-thresholding estimator with $\eta
_{n}\rightarrow 0$ and $n^{1/2}\eta _{n}\rightarrow \infty $. Assume that $%
\theta _{n}/\eta _{n}\rightarrow \zeta $ for some $\zeta \in {\mathbb{R}}%
\cup \{-\infty ,\infty \}$.

\begin{enumerate}
\item If $|\zeta |<1$, then $G_{H,n,\theta _{n}}$ converges weakly to
pointmass $\delta _{-\zeta }$.

\item If $|\zeta |=1$ and $n^{1/2}(\eta _{n}-\zeta \theta _{n})\rightarrow r$
for some $r\in \mathbb{R\cup \{-\infty },\mathbb{\infty \}}$, then $%
G_{H,n,\theta _{n}}$ converges weakly to%
\begin{equation*}
\Phi (r)\delta _{-\zeta }+(1-\Phi (r))\delta _{0}.
\end{equation*}

\item If $1<\left\vert \zeta \right\vert \leq \infty $, then $G_{H,n,\theta
_{n}}$ converges weakly to pointmass $\delta _{0}$.
\end{enumerate}
\end{theorem}

\begin{proof}
Consider case 1 first. On the event $\{\hat{\theta}_{H}=0\}$ we have $\eta
_{n}^{-1}(\hat{\theta}_{H}-\theta _{n})=-\eta _{n}^{-1}\theta _{n}$. By
Proposition \ref{selection_prob}, $P_{n,\theta _{n}}(\hat{\theta}%
_{H}=0)\rightarrow 1$. Since $\eta _{n}^{-1}\theta _{n}\rightarrow \zeta $
by assumption, the result follows. To prove case 2 write $\eta _{n}^{-1}(%
\hat{\theta}_{H}-\theta _{n})$ as $-\eta _{n}^{-1}\theta _{n}\boldsymbol{1}(%
\hat{\theta}_{H}=0)+(n^{1/2}\eta _{n})^{-1}Z_{n}\boldsymbol{1}(\hat{\theta}%
_{H}\neq 0)$ where $Z_{n}$ is standard normally distributed under $%
P_{n,\theta _{n}}$. Since Proposition \ref{selection_prob} shows that $%
P_{n,\theta _{n}}(\hat{\theta}_{H}=0)\rightarrow \Phi (r)$, the result in
case 2 now follows as is easily seen. To prove case 3, observe that $\eta
_{n}^{-1}(\hat{\theta}_{H}-\theta _{n})=(n^{1/2}\eta _{n})^{-1}n^{1/2}(\hat{%
\theta}_{H}-\theta _{n})$ and that $n^{1/2}(\hat{\theta}_{H}-\theta _{n})$
converges to a standard normal distribution under $P_{n,\theta _{n}}$ in
view of Theorem \ref{H_consist}.
\end{proof}

\begin{theorem}
\label{soft_rescaled}Consider the soft-thresholding estimator with $\eta
_{n}\rightarrow 0$ and $n^{1/2}\eta _{n}\rightarrow \infty $. Assume that $%
\theta _{n}/\eta _{n}\rightarrow \zeta $ for some $\zeta \in {\mathbb{R}}%
\cup \{-\infty ,\infty \}$. Then $G_{S,n,\theta _{n}}$ converges weakly to
pointmass $\delta _{-\limfunc{sign}(\zeta )\min (1,\left\vert \zeta
\right\vert )}$.
\end{theorem}

\begin{proof}
From (\ref{distri_S_1}) we obtain that 
\begin{equation*}
G_{S,n,\theta _{n}}(x)=\Phi (n^{1/2}\eta _{n}(x+1))\boldsymbol{1}(x\geq
-\theta _{n}/\eta _{n})+\Phi (n^{1/2}\eta _{n}(x-1))\boldsymbol{1}(x<-\theta
_{n}/\eta _{n}).
\end{equation*}%
Now it is easy to see that this expression converges to $0$ if $x<-\limfunc{%
sign}(\zeta )\min (1,\left\vert \zeta \right\vert )$ and to $1$ if $x>-%
\limfunc{sign}(\zeta )\min (1,\left\vert \zeta \right\vert )$.
\end{proof}

\begin{theorem}
\label{SCAD_rescaled}Consider the SCAD estimator with $\eta _{n}\rightarrow
0 $ and $n^{1/2}\eta _{n}\rightarrow \infty $. Assume that $\theta _{n}/\eta
_{n}\rightarrow \zeta $ for some $\zeta \in {\mathbb{R}}\cup \{-\infty
,\infty \}$.

\begin{enumerate}
\item If $|\zeta |\leq 2$, then $G_{SCAD,n,\theta _{n}}$ converges weakly to
pointmass $\delta _{-\limfunc{sign}(\zeta )\min (1,\left\vert \zeta
\right\vert )}$.

\item If $2<|\zeta |<a$, then $G_{SCAD,n,\theta _{n}}$ converges weakly to
pointmass $\delta _{-\limfunc{sign}(\zeta )(a-\left\vert \zeta \right\vert
)/(a-2)}$.

\item If $a\leq \left\vert \zeta \right\vert \leq \infty $, then $%
G_{SCAD,n,\theta _{n}}$ converges weakly to pointmass $\delta _{0}$.
\end{enumerate}
\end{theorem}

\begin{proof}
If $|\zeta |<1$ the proof is identical to the proof of case 1 in Theorem \ref%
{hard_rescaled}. Next assume $\zeta =1$: assume also for the moment that $%
n^{1/2}(\eta _{n}-\theta _{n})\rightarrow r$, $r\in {\mathbb{R}}\cup
\{-\infty ,\infty \}$. The atomic part $G_{0,n,\theta _{n}}$ of the cdf $%
G_{SCAD,n,\theta _{n}}(x)$ is given by $\{\Phi (n^{1/2}(-\theta _{n}+\eta
_{n}))-\Phi (n^{1/2}(-\theta _{n}-\eta _{n}))\}\boldsymbol{1}(x\geq -\theta
_{n}/\eta _{n})$ which is seen to converge weakly to $\Phi (r)\boldsymbol{1}%
(x\geq -1)$ which is the cdf of $\Phi (r)\delta _{-1}$. Furthermore,
recalling the definition of $F_{i,n,\theta }$ given in the proof of Theorem %
\ref{SCAD_consist}, 
\begin{eqnarray*}
G_{1,n,\theta _{n}}(x) &=&F_{1,n,\theta _{n}}(n^{1/2}\eta _{n}x) \\
&=&\int_{-\infty }^{n^{1/2}\eta _{n}(x+1)}\phi (z)\;\boldsymbol{1}\left(
n^{1/2}(\eta _{n}-\theta _{n})<z\leq n^{1/2}(2\eta _{n}-\theta _{n})\right)
\;dz
\end{eqnarray*}%
is seen to converge to $0$ for $x<-1$ and to $1-\Phi (r)$ for $x>-1$, since $%
n^{1/2}(\eta _{n}-\theta _{n})\rightarrow r$ and $n^{1/2}(2\eta _{n}-\theta
_{n})=n^{1/2}\eta _{n}(2-\theta _{n}/\eta _{n})\rightarrow \infty $. Hence, $%
G_{1,n,\theta _{n}}$ converges weakly to $(1-\Phi (r))\delta _{-1}$, and
thus $G_{0,n,\theta _{n}}+G_{1,n,\theta _{n}}$ converges weakly to pointmass 
$\delta _{-1}$. This implies that also $G_{SCAD,n,\theta _{n}}$ has the same
limit. Since the limit does not depend on $r$, a subsequence argument as in
the proof of Theorem \ref{SCAD_consist} completes the proof of the case $%
\zeta =1$. Next consider the case $1<\zeta <2$: Here $G_{1,n,\theta _{n}}(x)$
is easily seen to converge to $0$ for $x<-1$ and to $1$ for $x>-1$, since $%
n^{1/2}(\eta _{n}-\theta _{n})=n^{1/2}\eta _{n}(1-\theta _{n}/\eta
_{n})\rightarrow -\infty $ and $n^{1/2}(2\eta _{n}-\theta _{n})=n^{1/2}\eta
_{n}(2-\theta _{n}/\eta _{n})\rightarrow \infty $. Hence, $G_{1,n,\theta
_{n}}$ converges weakly to pointmass $\delta _{-1}$, and consequently $%
G_{SCAD,n,\theta _{n}}$ has to have the same limit. We turn to the case $%
\zeta =2$: Assume now for the moment that $n^{1/2}(2\eta _{n}-\theta
_{n})\rightarrow r$, $r\in {\mathbb{R}}\cup \{-\infty ,\infty \}$. Then $%
G_{1,n,\theta _{n}}(x)$ is seen to converge to $0$ for $x<-1$ and to $\Phi
(r)$ for $x>-1$, since $n^{1/2}(\eta _{n}-\theta _{n})=n^{1/2}\eta
_{n}(1-\theta _{n}/\eta _{n})\rightarrow -\infty $ and $n^{1/2}(2\eta
_{n}-\theta _{n})\rightarrow r$. Furthermore, note that in the case
considered $\left( (a-2)n^{1/2}\eta _{n}x+n^{1/2}(a\eta _{n}-\theta
_{n})\right) /(a-1)$ converges to $-\infty $ for $x<-1$ and to $\infty $ for 
$x>-1$. Consequently,%
\begin{equation}
\begin{split}
G_{2,n,\theta _{n}}(x)& =F_{2,n,\theta _{n}}(n^{1/2}\eta _{n}x) \\
& =\;\;\int_{-\infty }^{\left( (a-2)n^{1/2}\eta _{n}x+n^{1/2}(a\eta
_{n}-\theta _{n})\right) /(a-1)}\phi (z) \\
& \qquad \times \quad \boldsymbol{1}\left( n^{1/2}(2\eta _{n}-\theta
_{n})<z\leq n^{1/2}(a\eta _{n}-\theta _{n})\right) \,dz
\end{split}
\label{99}
\end{equation}%
is seen to converge to $0$ for $x<-1$ and to $1-\Phi (r)$ for $x>-1$, since $%
n^{1/2}(2\eta _{n}-\theta _{n})\rightarrow r$ and $n^{1/2}(a\eta _{n}-\theta
_{n})=n^{1/2}\eta _{n}(a-\theta _{n}/\eta _{n})\rightarrow \infty $. But
this shows that $G_{1,n,\theta _{n}}+G_{2,n,\theta _{n}}$ converges weakly
to pointmass $\delta _{-1}$, and hence the same must be true for $%
G_{SCAD,n,\theta _{n}}$. Since the limit does not depend on $r$, a
subsequence argument completes the proof for the case $\zeta =2$. Consider
next the case where $2<\zeta <a$: Then $G_{2,n,\theta _{n}}(x)$ is easily
seen to converge to $0$ if $x<-(a-\zeta )/(a-2)$ and to $1$ if $x>-(a-\zeta
)/(a-2)$, since $\left( (a-2)n^{1/2}\eta _{n}x+n^{1/2}(a\eta _{n}-\theta
_{n})\right) /(a-1)$ converges to $-\infty $ or $\infty $ depending on
whether $x$ is smaller or larger than $-(a-\zeta )/(a-2)$, and since $%
n^{1/2}(2\eta _{n}-\theta _{n})\rightarrow -\infty $ and $n^{1/2}(a\eta
_{n}-\theta _{n})\rightarrow \infty $. This proves that $G_{2,n,\theta _{n}}$%
, and hence $G_{SCAD,n,\theta _{n}}$, converges weakly to pointmass $\delta
_{-(a-\zeta )/(a-2)}$. Assume next that $\zeta =a$ and assume for the moment
that $n^{1/2}(a\eta _{n}-\theta _{n})\rightarrow r$, $r\in {\mathbb{R}}\cup
\{-\infty ,\infty \}$: Then the upper limit in the integral defining $%
G_{2,n,\theta _{n}}$ converges to $\infty $ if $x>0$ and to $-\infty $ if $%
x<0$. This is obvious if $\left\vert r\right\vert <\infty $, and follows
from rewriting the upper limit as $n^{1/2}\eta _{n}\left( (a-2)x+a-\theta
_{n}/\eta _{n}\right) /(a-1)$ if $\left\vert r\right\vert =\infty $.
Furthermore, the lower limit in the indicator function in (\ref{99})
converges to $-\infty $, while the upper limit converges to $r$. This shows
that $G_{2,n,\theta _{n}}$ converges weakly to $\Phi (r)\delta _{0}$.
Inspection of $G_{3,n,\theta _{n}}(x)=F_{3,n,\theta _{n}}(n^{1/2}\eta
_{n}x)=\int_{-\infty }^{n^{1/2}\eta _{n}x}\phi (z)\boldsymbol{1}\left(
z>n^{1/2}(a\eta _{n}-\theta _{n})\right) dz$ shows that this converges
weakly to $(1-\Phi (r))\delta _{0}$. Together this gives weak convergence of 
$G_{2,n,\theta _{n}}+G_{3,n,\theta _{n}}$, and hence of $G_{SCAD,n,\theta
_{n}}$, to pointmass $\delta _{0}$. Since the limit does not depend on $r$,
a subsequence argument again completes the proof of the case $\zeta =a$.
Suppose next that $a<\zeta \leq \infty $: Inspection of $G_{3,n,\theta _{n}}$
immediately shows that it (and hence also $G_{SCAD,n,\theta _{n}}$)
converges weakly to pointmass $\delta _{0}$. The remaining cases for $\zeta
\leq -1$ are proved completely analogous to the corresponding cases with
positive $\zeta $.
\end{proof}

Finally, we provide an impossibility result for the estimation of the finite
sample distributions $G_{H,n,\theta }$, $G_{S,n,\theta }$, and $%
G_{SCAD,n,\theta }$.

\begin{theorem}
\label{t3} Let $\hat{\theta}$ denote any one of the estimators $\hat{\theta}%
_{H}$, $\hat{\theta}_{S}$, or $\hat{\theta}_{SCAD}$, and write $G_{n,\theta
} $ for the cdf of $\eta _{n}^{-1}(\hat{\theta}-\theta )$ under $P_{n,\theta
}$. Let $0<\eta _{n}<\infty $ and let $t\in \mathbb{R}$ be arbitrary. Then 
\emph{every} estimator $\hat{G}_{n}(t)$ of $G_{n,\theta }(t)$ satisfies 
\begin{equation}
\sup_{|\theta |<c\eta _{n}}P_{n,\theta }\left( \left\vert \hat{G}%
_{n}(t)-G_{n,\theta }(t)\right\vert \;>\;\varepsilon \right) \quad \geq
\quad \frac{1}{2}  \label{t3.2}
\end{equation}%
for each $\varepsilon <(\Phi (n^{1/2}\eta _{n}(t+1))-\Phi (n^{1/2}\eta
_{n}(t-1)))/2$, for each $c>\left\vert t\right\vert $, and for each fixed
sample size $n$. If $\eta _{n}$ satisfies $\eta _{n}\rightarrow 0$ and $%
n^{1/2}\eta _{n}\rightarrow \infty $ as $n\rightarrow \infty $, we thus have
for each $c>\left\vert t\right\vert $%
\begin{equation}
\liminf_{n\rightarrow \infty }\inf_{\hat{G}_{n}(t)}\sup_{|\theta |<c\eta
_{n}}P_{n,\theta }\left( \left\vert \hat{G}_{n}(t)-G_{n,\theta
}(t)\right\vert \;>\;\varepsilon \right) \quad \geq \quad \frac{1}{2}
\label{t3.1}
\end{equation}%
for each $\varepsilon <1/2$ if $\left\vert t\right\vert <1$ and for $%
\varepsilon <1/4$ if $\left\vert t\right\vert =1$, where the infimum in (\ref%
{t3.1}) extends over \emph{all} estimators $\hat{G}_{n}(t)$.
\end{theorem}

This result shows, in particular, that no uniformly consistent estimator
exists for $G_{n,\theta }(t)$ in case $\left\vert t\right\vert \leq 1$ (not
even over compact subsets of $\mathbb{R}$ containing the origin). In view of
Theorems \ref{hard_rescaled}, \ref{soft_rescaled}, and \ref{SCAD_rescaled}
we see that for $t>1$ we have $\sup_{\theta \in \mathbb{R}}\left\vert
G_{n,\theta }(t)-1\right\vert \rightarrow 0$ as $n\rightarrow \infty $,
hence $\hat{G}_{n}(t)=1$ is trivially a uniformly consistent estimator.
Similarly, for $t<-1$ we have $\sup_{\theta \in \mathbb{R}}\left\vert
G_{n,\theta }(t)\right\vert \rightarrow 0$ as $n\rightarrow \infty $, hence $%
\hat{G}_{n}(t)=0$ is trivially a uniformly consistent estimator.

\begin{proof}
We first prove (\ref{t3.2}). For fixed $n$ and $t$ set $s=n^{1/2}\eta _{n}t$%
. Define $\hat{F}_{n}(s)=\hat{G}_{n}(t)$. Also note that $G_{n,\theta
}(t)=F_{n,\theta }(s)$ holds. By Theorem \ref{t2} we know that%
\begin{equation*}
\sup_{|\theta |<d/n^{1/2}}P_{n,\theta }\left( \left\vert \hat{F}%
_{n}(s)-F_{n,\theta }(s)\right\vert \;>\;\varepsilon \right) \quad \geq
\quad \frac{1}{2}
\end{equation*}%
for each $\varepsilon <(\Phi (s+n^{1/2}\eta _{n})-\Phi (s-n^{1/2}\eta
_{n}))/2$ and for each $d>\left\vert s\right\vert $. Rewriting this in terms
of $t$, $\hat{G}_{n}(t)$, and $G_{n,\theta }(t)$ and setting $%
c=dn^{-1/2}/\eta _{n}$ gives (\ref{t3.2}). Relation (\ref{t3.1}) is a
trivial consequence of (\ref{t3.2}).
\end{proof}

\section*{References}

\quad \thinspace \thinspace

Bauer, P., P\"{o}tscher, B.~M. \& P. Hackl (1988): Model selection by
multiple test procedures. \emph{Statistics \ }19, 39--44.

Beran, R. (1997): Diagnosing bootstrap success. \emph{Annals of the
Institute of Statistical Mathematics \ }49, 1-24.

Bruce, A.~G. \& H. Gao (1996): Understanding WaveShrink: Variance and bias
estimation. \emph{Biometrika \ }83, 727-745.

Efron, B., Hastie, T., Johnstone, I. \& R. Tibshirani (2004): Least angle
regression. \emph{Annals of Statistics } 32, 407--499.

Fan, J. \& R. Li (2001): Variable selection via nonconcave penalized
likelihood and its oracle properties. \emph{Journal of the American
Statistical Association} \ 96, 1348-1360.

Frank, I.~E. \& J.~H. Friedman (1993): A statistical view of some
chemometrics regression tools (with discussion). \emph{Technometrics \ }35,
109-148.

Kabaila, P. (1995): The effect of model selection on confidence regions and
prediction regions. \emph{Econometric Theory} {\ 11}, 537--549.

Knight, K. \& W. Fu (2000): Asymptotics for lasso-type estimators. \emph{%
Annals of Statistics \ }28, 1356-1378.

Knight, K. (2008): Shrinkage estimation for nearly-singular designs. \emph{%
Econometric Theory \ }24, forthcoming.

Kulperger, R.~J. \& S.~E.\ Ahmed (1992): A bootstrap theorem for a
preliminary test estimator. \emph{Communications in Statistics: Theory and
Methods }\ 21, 2071--2082.

Judge, G.~G. \& M.~E. Bock (1978): \emph{The Statistical Implications of
Pre-test and Stein-Rule Estimators in Econometrics}. North-Holland.

Leeb, H. \& B.~M. P\"{o}tscher (2003): The finite-sample distribution of
post-model-selection estimators and uniform versus nonuniform
approximations. \emph{Econometric Theory} {\ 19}, 100--142.

Leeb, H. \& B.~M. P\"{o}tscher (2005): Model selection and inference: Facts
and fiction. \emph{Econometric Theory} {\ 21}, 21--59.

Leeb, H. \& B.~M. P\"{o}tscher (2006a): Performance limits for estimators of
the risk or distribution of shrinkage-type estimators, and some general
lower risk-bound results. \emph{Econometric Theory} {\ 22}, 69--97.
(Corrigendum: ibidem, 24, 581-583).

Leeb, H. \& B.~M. P\"{o}tscher (2006b): Can one estimate the conditional
distribution of post-model-selection estimators? \emph{Annals of Statistics
\ }34, 2554-2591.

Leeb, H. \& B.~M. P\"{o}tscher (2008a): Sparse estimators and the oracle
property, or the return of Hodges' estimator. \emph{Journal of Econometrics}
142, 201-211.

Leeb, H. \& B.~M. P\"{o}tscher (2008b): Can one estimate the unconditional
distribution of post-model-selection estimators? \emph{Econometric Theory \ }%
24, 338-376.

Lehmann, E.~L. \& G. Casella (1998): \emph{Theory of Point Estimation}.
Springer Texts in Statistics. New York: Springer-Verlag.

P\"{o}tscher, B.~M. (1991): Effects of model selection on inference. \emph{%
Econometric Theory} {\ 7}, 163--185.

P\"{o}tscher, B.~M. (2006): The distribution of model averaging estimators
and an impossibility result regarding its estimation. \emph{IMS Lecture
Notes--Monograph Series} \ 52, 113-129.

P\"{o}tscher, B.~M. \& U. Schneider (2009): On the distribution of the
adaptive lasso estimator. \emph{Journal of Statistical Planning and Inference%
}, forthcoming, doi:10.1016/j.jspi.2009.01.003.

Samworth, R. (2003): A note on methods of restoring consistency of the
bootstrap. \emph{Biometrika \ }90, 985--990.

Sen, P.~K. (1979): Asymptotic properties of maximum likelihood estimators
based on conditional specification. \emph{Annals of Statistics \ }7,
1019-1033.

Tibshirani, R. (1996): Regression shrinkage and selection via the lasso.\ 
\emph{Journal of the Royal Statistical Society Series B} \ 58, 267-288.

Zhao, P. \& B. Yu (2006): On model selection consistency of lasso. \emph{%
Journal of Machine Learning Research \ }7, 2541-2563.

Zou, H. (2006): The adaptive lasso and its oracle properties. \emph{Journal
of the American Statistical Association} \ 101, 1418-1429.

\end{document}